\documentclass[11pt]{article}   
\usepackage[greek,english]{babel}  
\usepackage{graphicx}
\usepackage{amsmath}
\usepackage{amssymb}
\usepackage{amsmath,amsthm}
\usepackage{verbatim}
\usepackage{sidecap}
\usepackage{mathrsfs}
\usepackage{faktor}
\usepackage{tikz}
\usepackage{color}
\newtheorem{theorem}{Theorem}[section]
\newtheorem{definition}[theorem]{Definition}
\newtheorem{proposition}[theorem]{Proposition}

\newtheorem{lemma}[theorem]{Lemma}

\theoremstyle{definition}
\newtheorem{remark}[theorem]{Remark}

\def\ve{\varepsilon}

\titlepage
\title{\bf Periodic solutions and regularization \\of a Kepler problem with time-dependent perturbation}
\author{Alberto Boscaggin, Rafael Ortega and Lei Zhao}
\date{}

\begin{document}
\maketitle

\begin{abstract}
We consider a Kepler problem in dimension two or three, with a time-dependent $T$-periodic perturbation.  
We prove that for any prescribed positive integer $N$, there exist at least $N$ periodic solutions (with period $T$) as long as the perturbation is small enough. Here the solutions are understood in a general sense as they can have collisions. The concept of generalized solutions is defined intrinsically and it coincides with the notion obtained in Celestial Mechanics via the theory of regularization of collisions.
\end{abstract}

\noindent
{\footnotesize \textbf{AMS-Subject Classification}}. {\footnotesize 37J45, 70F05, 70F16.}


\section{Introduction}

The theory of perturbations of the Kepler problem is a classical chapter in Celestial Mechanics. See for instance \cite{Mou} and \cite{Cor}. In this paper we consider non-autonomous perturbations which are periodic in time and preserve the Newtonian structure. More precisely, we shall consider the system 
\begin{equation}\label{eqmain0}
\ddot u = -\frac{u}{\vert u \vert^3} + \varepsilon \,\nabla_u U(t,u,\varepsilon), \qquad u \in \mathbb{R}^N,
\end{equation}
where $N = 2$ or $3$, $\varepsilon$ is a small real parameter and the force function $U$ is smooth and $T$-periodic in the variable $t$. We will prove that, for any prescribed integer $l \geq 1$ and $\varepsilon$ small enough, there exist at least $l$ periodic solutions with period $T$. These solutions will be understood in a generalized sense because collisions cannot be excluded: a solution $u = u(t)$ will be continuous everywhere and will satisfy the equation \eqref{eqmain0} whenever $u(t) \neq 0$. In addition the energy $\tfrac12 \vert \dot u(t) \vert^2 - \tfrac{1}{\vert u(t) \vert}$ and the direction $\tfrac{u(t)}{\vert u(t) \vert}$ will have a finite limit at collisions; that is, as $t \to t_0$ with $u(t_0) = 0$. 
Experts in Celestial Mechanics will probably prefer to say that a function $u(t)$ is a solution of \eqref{eqmain0} if it can be transformed into a solution of some regularized problem. Later we will see that, at least in two dimensions, 
both notions are indeed equivalent. We prefer the above definition because it is intrinsic and does not require an a priori knowledge of the techniques of regularization of collisions.

Problem \eqref{eqmain0} looks like a perturbation of a completely integrable Hamiltonian system.
When the unperturbed problem possesses a $N$-dimensional torus of $T$-periodic solutions, a well established theory (see for instance \cite{ACE87}) ensures that, under a suitable non-degeneracy condition, at least $N+1 = \textnormal{Cat}(\mathbb{T}^N)$ solutions with period $T$ survive when $\varepsilon$ is small enough. Due to its peculiar degeneracies, however, this does not apply to the Kepler problem. The ultimate reason for this is Kepler's third law: the period of a Keplerian ellipse depends only on its major semi-axis, thus giving rise to a larger set of $T$-periodic solutions for $\varepsilon = 0$. Even worst, to obtain a compact manifold of solutions of the unperturbed problem, one is forced to take into account also rectilinear motions and to allow collisions. 
To overcome these difficulties, a first possibility is to abandon this global point of view. In this case, the averaging method becomes one of the most classical technique to prove the existence of periodic solutions of \eqref{eqmain0}. After changing the system to appropriate coordinates (say that of Delaunay or Poincar\'e), the critical points of an averaged force function $U_{\#}$ obtained from $U$ produce branches of periodic solutions for small $\varepsilon$. This is a perturbative method and solutions are without collisions if we start from a circular or an elliptic motion at $\varepsilon = 0$. The disadvantage of this method is that it imposes additional conditions on the function $U$ which are usually found after long computations. 
On the other hand, one can try to recover a more global approach by imposing some symmetry condition on $U$.
For instance, variational techniques have been employed by Ambrosetti and Coti Zelati \cite{AC89,AC93} and by Serra and Terracini \cite{ST}; in both these works the assumptions on $U$ are used in order to guarantee that the critical points of the action functional do not have collisions. See also \cite{FonGal} for additional references. Our result is of different nature because we are concerned with rather general perturbations and collisions are allowed.

The basic tools in this paper will be regularization theory and bifurcation from periodic manifolds. As it is traditional we transform the periodic system \eqref{eqmain0} into an autonomous Hamiltonian system of the type
\begin{equation}\label{primo}
\dot u = \partial_v \mathcal{H}_\varepsilon, \qquad \dot v = -\partial_u \mathcal{H}_\varepsilon, \qquad \dot t = \partial_\tau \mathcal{H}_\varepsilon, \qquad \dot \tau = -\partial_t \mathcal{H}_\varepsilon
\end{equation}
with $\mathcal{H}_\varepsilon = \mathcal{H}_\varepsilon(u,v,t,\tau)$. This system has $N+1$ degrees of freedom and the coordinates in the phase space are $u \in \mathbb{R}^N \setminus \{0\}$, $v = \dot u \in \mathbb{R}^N$, $t \in \mathbb{R}/T\mathbb{Z}$ and $\tau \in \mathbb{R}$, where $\tau$ is the conjugate variable of the time $t$. The dynamics of \eqref{eqmain0} outside the singularity is reproduced by \eqref{primo} at the energy level $\mathcal{H}_\varepsilon^{-1}(0)$. After applying Levi-Civita and Kustaanheimo-Stiefel regularization ($N=2$ or $3$) we obtain a symplectic manifold $(M,\omega)$ of dimension $2N+2$ and a Hamiltonian system
\begin{equation}\label{secondo}
\dot x = X_{\mathcal{K}_\varepsilon}(x), \qquad x \in \mathcal{K}_\varepsilon^{-1}(0) \subset M,
\end{equation}
which can be understood as the regularization of the non-autonomous system \eqref{eqmain0}. The Hamiltonian functions $\mathcal{H}_\varepsilon$ and $\mathcal{K}_\varepsilon$ are related but they are not equivalent. From our point of view the key fact on regularization is that the closed orbits of \eqref{secondo} produce generalized periodic solutions of \eqref{eqmain0}. In principle these solutions can be sub-harmonic with period $\eta T$ for some $\eta = 1,2,\ldots$ but the topology of $M$ allows to define a winding number for the closed orbits of \eqref{secondo} that coincides with $\eta$. We must look for closed orbits of \eqref{secondo} with winding number $\eta = 1$.

The next step is the analysis of the unperturbed problem ($\varepsilon = 0$). A direct computation allows to construct a sequence of compact and connected manifolds
$$
P_n \subset \mathcal{K}_0^{-1}(0) \subset M, \qquad n=1,2,\ldots
$$
which are filled by closed orbits of \eqref{secondo} with $\eta = 1$ and $\varepsilon = 0$. These manifolds have dimension $2N$, in essence they correspond to the periodic solutions of the Kepler problem with minimal period $\tfrac{T}{n}$. Note that circular, elliptic and rectilinear motions are included.

The proof will be complete if we show that each manifold $P_n$ admits bifurcation for positive $\varepsilon$. Then every bifurcating branch will produce a periodic solution of the original system \eqref{eqmain0} if $\varepsilon > 0$ is small enough.

To prove the existence of these bifurcations we will apply a result of Weinstein in \cite{Wei73}. This is a delicate step of the proof because most of the results in the literature on periodic manifolds do not apply to our problem (see Remark \ref{non-deg} for more details). Fortunately the non-degeneracy condition in \cite{Wei73} is very weak and the manifold $P_n$ satisfies it.   

The general strategy described above works well for dimensions $N=2$ and $N=3$ but in the three-dimensional case some additional subtleties appear, mainly related to the use of Kustaanheimo-Stiefel regularization. Compared to Moser regularization \cite{Mos70}, Kustaanheimo-Stiefel regularization transforms the unperturbed system into an integrable system defined on a Euclidean space, which is thus easier for our current purpose of calculation. However it carries additional symmetry and thus in this case $M$ is obtained as a quotient manifold via symplectic reduction. For a discussion about the relationship between Kustaanheimo-Stiefel and Moser regularization, see \cite{Kum}.

The rest of the paper is organized in five sections. First we present a short introduction to Weinstein's theory in Section \ref{sec2}. The main result of the paper has been already discussed although in very unprecise terms. The formal statement can be found in Section \ref{sec3}. The proof of this theorem for $N=2$ is found in Section \ref{sec4}. This section also contains a discussion on the equivalence of the possible definitions of generalized solution. Section \ref{sec5} contains the proof of the theorem for $N=3$. The last section of the paper is concerned with other related results on this problem: bifurcation from infinity, removal of collisions by small changes in the function $U$and a model considered by Fatou in \cite{F1928}.

\section{Bifurcation from a periodic manifold}\label{sec2}

In \cite{Wei73} Weinstein considered a Hamiltonian system having a continuum of periodic orbits and explained how to obtain bifurcations from this continuum. This technique will be essential in our proof and we are going to present a short description. The terminology is taken from \cite{Wei73} with a view towards our precise needs.

Let $(M,\omega)$ be a symplectic manifold and let $H \in \mathcal{C}^\infty(M)$ be a function on $M$. The corresponding Hamiltonian vector field will be denoted by $X_H$. The solution of the system
\begin{equation}\label{ham}
\dot x = X_H(x)
\end{equation}
satisfying $x(0) = \xi$ is denoted by $\phi_t(\xi)$.

Assume that the number $E \in \mathbb{R}$ is a regular value of $H$ so that $H^{-1}(E)$ is a submanifold of $M$. Given a non-constant periodic solution $x(t)$ of \eqref{ham} with $H(x(t)) = E$, we fix a period $\tau > 0$ and consider the linear map
$$
P: T_p M \to T_p M, \qquad P(v) = \partial_\xi \phi_\tau(p)v, 
$$
where $p = x(0)$. Sometimes $P$ is called the monodromy operator at $p$. It is well known (see for instance \cite{MosZeh}) that the vector $w = X_H(p)$ and the hyperplane $W = T_p(H^{-1}(E))$ are invariant under $P$. Moreover, $w \in W$. 

The invariance of $W$ allows to restrict the monodromy operator to this sub-space. This restriction will be denoted by $P_W: W \to W$. Let $L_w$ be the one-dimensional subspace of $W$ spanned by $w$,
$$
L_w = \{ \lambda w \, : \, \lambda \in \mathbb{R}\}.
$$
We will be interested in the dimension of the space
$$
\mathcal{E} = \left( \textnormal{Id} - P_W \right)^{-1}(L_w).
$$
This is a subspace of $W$ containing the eigenspace $\textnormal{Ker}\left( \textnormal{Id} - P_W \right)$. In particular $w \in \mathcal{E}$ and so $\mathcal{E}$ has dimension at least one. The integer $\textnormal{dim}\,\mathcal{E}-1$ can be interpreted as a degeneracy index for the periodic solution. To explain this we translate the previous discussion to the language of matrices. Let us select a basis $\{v_1,v_2,\ldots,v_{2N}\}$ of $T_p(M)$ such that $v_1 \notin W$, $v_2 = w$ and $v_3,\ldots,v_{2N} \in W$. The matrix associated to $P$ is of the type
$$
\mathcal{M}_p = \left( \begin{array}{cc}
1 & 0  \\
\alpha & \mathcal{M}_W
\end{array}\right), \qquad \mbox{ with } \quad 
\mathcal{M}_W = \left( \begin{array}{cc}
1 & \beta  \\
0 & \Gamma 
\end{array}\right).
$$
The sub-matrix $\Gamma$ has dimension $(2N-2) \times (2N-2)$ and it contains all the information concerning $\textnormal{dim}\,\mathcal{E}$. Actually,
$$
\textnormal{dim}\,\mathcal{E} = 1 + \textnormal{dim}\, \textnormal{ker} \left( \textnormal{Id} - \Gamma \right)
$$
and the isomorphism theorem implies that
$$
2N - 2 = \textnormal{dim}\, \textnormal{ker} \left( \textnormal{Id} - \Gamma \right) +  \textnormal{rank} \left( \textnormal{Id} - \Gamma \right)
$$
so that
\begin{equation}\label{isoth}
\textnormal{dim}\,\mathcal{E} - 1 = 2N - 2 - \textnormal{rank} \left( \textnormal{Id} - \Gamma \right).
\end{equation}
We observe that the number $\textnormal{dim}\,\mathcal{E}$ only depends on the periodic solution $x(t)$ and the chosen period $\tau > 0$. Moreover this number is invariant under symplectic diffeomorphisms.

Given $H$ and $E$ in the previous conditions we define the set of periodic points with energy $E$, denoted by $\textnormal{Per}_H^E$, as the set of couples $(p,\tau) \in M \times \,]0,\infty[$ satisfying
$$
H(p) = E, \qquad X_H(p) \neq 0, \qquad \phi_\tau(p) = p.
$$

Note that $\phi_t(p)$ is a non-constant periodic solution with period $\tau$. The same closed orbit will produce other points lying in 
$\textnormal{Per}_H^E$ since all pairs $(p,N\tau)$ lie also in $\textnormal{Per}_H^E$ if $N \geq 1$ is an integer.

A subset $\Sigma \subset \textnormal{Per}_H^E$ is a \emph{periodic manifold} if it satisfies the conditions
\begin{itemize}
\item[(i)] $\Sigma$ is a closed submanifold of $M \times \mathbb{R}$,
\item[(ii)] The restriction to $\Sigma$ of the projection $\pi: M \times \mathbb{R} \to M$ is an embedding.
\end{itemize}

We will say that the periodic manifold is \emph{non-degenerate} if it also satisfies the condition below,
\begin{itemize}
\item[(iii)] The tangent space of $\pi(\Sigma)$ coincides with $\mathcal{E}$ for each $(p,\tau) \in \Sigma$.
\end{itemize}
The last condition was stated in a different but equivalent way in \cite{Wei73}. In fact Lemma 1.1 proves the equivalence of (NDMP3) in \cite{Wei73} and our condition (iii).

We are ready to state a corollary of Theorem 1.4 in \cite{Wei73}. 

Given a symplectic manifold $(M,\omega)$ such that the form $\omega$ is exact on $M$, a function $H \in \mathcal{C}^\infty(M)$ and a regular value $E \in \mathbb{R}$, we assume that $\Sigma$ is a compact non-degenerate periodic manifold. In addition $\mathcal{U}$ is a neighborhood of $\Sigma$ in $M \times \mathbb{R}$ and $\mathcal{H} = \mathcal{H}(x,\ve)$ is a function in $\mathcal{C}^\infty(M \times [0,1])$ with $\mathcal{H}(\cdot,0) = H$. Then there exists $\ve_0 > 0$ such that for each $\ve \in \,]0,\ve_0[$ the system
$$
\dot x = X_{\mathcal{H}(\cdot,\ve)}(x), \qquad \mathcal{H}(x,\ve) = E
$$
has at least $m$ closed orbits lying in $\mathcal{U}$, where $m$ is the least integer greater or equal than 
$\textnormal{Cat}(\Sigma)/2$ (here $\textnormal{Cat}(\Sigma)$ denotes the Lusternik-Schnirelman category of $\Sigma$).

In practice, to check the condition of non-degeneracy (iii) it is convenient to employ the degeneracy index defined above. To explain this we need some preliminary remarks.

Given a periodic manifold, the tangent space is always contained in $\mathcal{E}$, that is
\begin{equation}\label{ninc}
T_p(\pi(\Sigma)) \subset \mathcal{E}, \quad \mbox{ for each } (p,\tau) \in \Sigma.
\end{equation}
To prove this we assume that $\delta$ is a vector in $T_{p}(\pi(\Sigma))$ and we consider a smooth path
$p_s$ in $\pi(\Sigma)$, $s \in \,]-\ve,\ve[\,$, passing through $p$ at $s = 0$ and such that
$$
\frac{d}{ds} p_s \vert_{s = 0} = \delta.
$$
This path can be lifted to a smooth path $(p_s,\tau_s)$ is $\Sigma$ with $\tau_0 = \tau$;
moreover, since the path lies in $\textnormal{Per}_H^E$,
$$
p_s = \phi_{\tau_s}(p_s).
$$
Differentiating with respect to $s$ and letting $s = 0$ we obtain
$$
\delta = \partial_\xi \phi_{\tau}(p)\delta + \sigma X_H(p),
$$
where $\sigma = \frac{d}{ds} \tau_s \vert_{s = 0}$. In the above notations,
$$
\left( \textnormal{Id} - P \right) \delta = \sigma w.
$$
The curve $s \mapsto p_s$ lies in $H^{-1}(E)$ and so $\delta \in T_p(H^{-1}(E)) = W$. Summing up the previous discussion, 
$\delta \in \left( \textnormal{Id} - P_W \right)^{-1}(L_w) = \mathcal{E}$.

Once we know that \eqref{ninc} holds we observe that (iii) is equivalent to $\textnormal{dim}(\Sigma) = \textnormal{dim}(\mathcal{E})$. Let us now assume that $\Sigma$ is a periodic manifold of dimension $2N-2$. From the identity \eqref{isoth} we deduce that $\Sigma$ is non-degenerate if and only if $\Gamma$ is not the identity matrix.

\section{Main result}\label{sec3}

We consider the perturbed Kepler problem \eqref{eqmain0}
where $N=2$ or $N = 3$ and the force function $U: \mathbb{R} \times \mathbb{R}^N \times [0,1] \to \mathbb{R}$ is
$\mathcal{C}^\infty$ and satisfies the periodicity condition
$$
U(t+T,u,\varepsilon) = U(t,u,\varepsilon)
$$
for some fixed $T > 0$.

\begin{theorem}\label{main}
Given an integer $l \geq 1$, there exists $\varepsilon_* = \varepsilon_*(l) > 0$ such that the equation \eqref{eqmain0} has at least $l$ generalized $T$-periodic solutions for each $\varepsilon \in \,]0,\varepsilon_*[\,$.
\end{theorem}

It must be noticed that these solutions can have collisions, meaning that $u(t)$ can take the value $u = 0$ at some instants. This forces us to be precise on the notion of solution that is employed.

\begin{definition}\label{def: bs}
A \emph{generalized solution} to \eqref{eqmain0} is a continuous function $u: \mathbb{R} \to \mathbb{R}^N$ satisfying the following conditions:
\begin{itemize}
\item[(i)] the set $Z= \{t \in \mathbb{R} : u(t) = 0\}$ of collisions is discrete,
\item[(ii)] for any open interval $I \subset \mathbb{R} \setminus Z$, the function $u$ is $\mathcal{C}^\infty(I)$ and satisfies \eqref{eqmain0} on $I$, 
\item[(iii)] for any $t_0 \in Z$, the limits
$$
\lim_{t \to t_0} \frac{u(t)}{\vert u(t) \vert} \qquad \mbox{and} \qquad \lim_{t \to t_0}\left( \frac12 \vert \dot u(t) \vert^2- \frac{1}{\vert u(t) \vert}\right)
$$
of collision direction and collision energy exist and are finite.
\end{itemize}
\end{definition}

It is worth noticing that in the case $N=1$ the above definition reduces to the one given in \cite{Or}, asking for the preservation of the energy at the collisions. In the higher dimensional setting the preservation of the collision direction also plays a role. In particular for 
$N=2$ these two requirements together make the above notion of generalized solution equivalent to the one obtained via Levi-Civita regularization (see Section \ref{sec41} below). We suspect that for $N=3$ an equivalence with Kustaanheimo-Stiefel regularization can be proved, but we have not investigated this in detail.

\section{Proof in the 2-d case}\label{sec4}

In this section, we give the proof of Theorem \ref{main} for $N = 2$. We split our arguments in some steps.

\subsection{The regularized system}\label{sec41}

We follow the approach taken in \cite{Zha16}. That paper dealt with a Kepler problem in one degree of freedom ($N=1$) but the same strategy applies to $N=2$.

To start with we consider the system \eqref{eqmain0} understood in a classical sense ($u \in \mathbb{R}^2 \setminus \{0\}$). It can be transformed into the Hamiltonian system
$$
\dot u = \partial_v H_\varepsilon(t,u,v), \qquad \dot v = -\partial_u H_\varepsilon(t,u,v)
$$
with
$$
H_\varepsilon(t,u,v) = \frac12 \vert v \vert^2 - \frac{1}{\vert u \vert} - \varepsilon \,U(t,u,\varepsilon), \qquad 
(u,v) \in (\mathbb{R}^2 \setminus \{0\}) \times \mathbb{R}^2.
$$
Following a traditional approach we embed this time-periodic system into an autonomous Hamiltonian system with an additional degree of freedom. To do this we introduce a new unknown $\tau$ and consider the phase space
$(\mathbb{R}^2 \setminus \{0\}) \times \mathbb{R}^2 \times \mathbb{T} \times \mathbb{R}$ with $\mathbb{T} = \mathbb{R} / T\mathbb{Z}$ and coordinates $(u,v,t,\tau)$. The associated symplectic form is
$$
\sum_{i=1}^2 du_i \wedge dv_i + dt \wedge d\tau
$$
and the Hamiltonian function
$$
\mathcal{H}_\ve(u,v,t,\tau) = \tau + \frac{\vert v \vert^2}{2} - \frac{1}{\vert u \vert} - \ve \,U(t,u,\ve).
$$
Next we consider the canonical map associated to Levi-Civita change of variables
$$
LC: (\mathbb{C} \setminus \{0\}) \times \mathbb{C} \to (\mathbb{C} \setminus \{0\}) \times \mathbb{C},
\qquad LC(z,w) = \left(z^2,  \frac{w}{2 \bar{z}}\right) = (u,v).
$$
From now one $\mathbb{C}$ and $\mathbb{R}^2$ will be identified, although differentiability will be understood in a real sense. The Hamiltonian function $\widehat{\mathcal{H}}_\ve = \mathcal{H}_\ve \circ LC$ is defined in the same phase space, 
$\widehat{\mathcal{H}}_\ve = \widehat{\mathcal{H}}_\ve(z,w,t,\tau)$.

To eliminate the singularity at $z = 0$ we multiply by $\vert z \vert^2$ to obtain the new Hamiltonian function
$$
\mathcal{K}_\ve(z,w,t,\tau) = \tau \vert z \vert^2 + \frac{\vert w \vert^2}{8} - 1 - \ve P(t,z,\ve)
$$
with $P(t,z,\ve) = \vert z \vert^2 U(t,z^2,\ve)$.

To write the associated system we use the symbol $\nabla_z P(t,z,\ve)$ for the complex-valued function 
$$
\nabla_z P(t,z,\ve) = \partial_{x} P(t,z,\ve) + i \partial_{y} P(t,z,\ve)
$$ 
with $z = x + iy$. The Hamiltonian system is
\begin{equation}\label{lc0}
\left\{
\begin{array}{l}
\vspace{0.2cm}
\displaystyle z' = \frac{w}{4}\\
\vspace{0.2cm}
w' = - 2 \tau z + \ve \, \nabla_z P(t,z,\ve)\\
\vspace{0.2cm}
t' = \vert z \vert^2 \\
\tau' = \varepsilon \, \partial_t P(t,z,\ve)
\end{array}
\right.
\end{equation}
which we can be written in vector notation as
\begin{equation}\label{lc}
JX' = \nabla \mathcal{K}_\ve(X), 
\end{equation}
where $X = (z,w,t,\tau)^*$ and $J$ is the $6 \times 6$ matrix (described in $2 \times 2$ blocks)
$$ 
J = \left( \begin{array}{ccc}
0 & -\textnormal{Id}_{2} & 0  \\
\textnormal{Id}_{2} & 0 & 0  \\
0 & 0 & J_2
\end{array}\right), \qquad \mbox{ with } \quad 
J_2 = \left( \begin{array}{cc}
0 & -1  \\
1 & 0 
\end{array}\right).
$$
Now the phase space can be enlarged to $\mathbb{C}^2 \times \mathbb{T} \times \mathbb{R}$.

At this point it is important to remark that these successive changes of Hamiltonian functions do not produce equivalent systems. When passing from $\mathcal{H}_\ve$ to $\widehat{\mathcal{H}}_\ve$ we have employed the map $LC$ that is canonical but not one-to-one. Later we have passed from $\widehat{\mathcal{H}}_\ve$ to $\mathcal{K}_\ve = \vert z \vert^2 \widehat{\mathcal{H}}_\ve$. At the energy level $\widehat{\mathcal{H}}_\ve = 0$ the two systems have the same orbits and they could be thought as equivalent. However this is not exact because we have enlarged the phase space to include the collision set $z = 0$.

The smooth system \eqref{lc0} is defined on the six dimensional manifold $\mathbb{C}^2 \times \mathbb{T} \times \mathbb{R}$ and we are interested in the energy level $\mathcal{K}_\ve = 0$. We want to relate the closed orbits of this system 
with the generalized periodic solutions of system \eqref{eqmain0}. First we introduce the notion of \emph{index} of a closed orbit. Given a non-constant zero-energy periodic solution $X(s) = (z(s),w(s),t(s),\tau(s))$ of \eqref{lc0}, with minimal period $S > 0$, we lift its angular component to a smooth function $t: \mathbb{R} \to \mathbb{R}$ satisfying $t(s+S) = t(s) + \eta T$ for some $\eta \in \mathbb{Z}$. We call this integer $\eta$ the index of the solution $X(s)$. Notice that, in view of $t' = \vert z \vert^2$, we have $\eta \geq 1$.

\begin{lemma}\label{l1}
For any $\varepsilon > 0$, any periodic solution $X(s)$ of \eqref{lc0} lying on the energy level $\mathcal{K}_\ve^{-1}(0)$ gives rise to a generalized $\eta T$-periodic solution of \eqref{eqmain0}, where $\eta$ is the index of $X(s)$. 
\end{lemma}

\begin{proof}
Since $\eta \geq 1$, it is immediately seen that the lifting of the angular component of $X(s)$ (still denoted by $t(s)$) is a global (increasing) homeomorphism of $\mathbb{R}$; accordingly, we can consider its inverse function $s(t)$ and define	
$$
u(t) = z^2(s(t)), \qquad t \in \mathbb{R}.
$$
Notice that $s(t+ \eta T) = s(t) + S$; as a consequence, the function $u(t)$ is $\eta T$-periodic. 
The perturbation $P$ vanishes at $z = 0$ and so from $\mathcal{K}_\ve(X(s)) = 0$ we deduce that $\vert z'(s) \vert^2 = \tfrac12$ if $z(s) = 0$. In particular the zeros of $z(s)$ are non-degenerate and so the set of zeros of $u(t)$ is discrete, thus proving (i) in the definition of bouncing solutions.

As for (iii), the existence of the first limit follows from the fact that
$$
\frac{u(t)}{\vert u(t) \vert} = \left( \frac{z(s(t))}{\vert z(s(t))\vert }\right)^2
$$
and that the zeros of $z(s)$ are simple. The existence of the second limit is a consequence of the equality, valid
in $\mathbb{R} \setminus Z$,
$$
\frac12 \vert \dot u \vert^2 - \frac{1}{\vert u \vert}  = \frac{1}{\vert z \vert^2} \left(\frac{\vert w \vert^2}{8} -1 \right) \\
= - \tau + \ve \, U(t,z^2,\ve)
$$
and of the fact that the right hand-side is a continuous function in all its variables. Notice that in the above computations we have used the fact that $X(s)$ lies on the zero-set of $\mathcal{K}_\ve$. 

Finally, to prove that the condition (ii) in Definition \ref{def: bs} holds we compute, using both the differential equation and the zero-energy relation, 
\begin{align*}
\ddot u & = 
 2 \,\frac{z z'' \vert z \vert^2 - z^2 \vert z' \vert^2}{\vert z \vert^6}  \\
 & = -\frac{\tau \vert z \vert^2 z^2 - \frac{\varepsilon}{2} \vert z \vert^2 z \nabla_z P(t,z,\ve) + \frac{\vert w \vert^2}{8} z^2}{\vert z \vert^6}\\
 &  = - \frac{z^2}{\vert z \vert^6} - \varepsilon \frac{z^2 P(t,z,\ve) - \frac{\vert z \vert^2}{2} z\nabla_z P(t,z,\ve)}{\vert z \vert^6}.
\end{align*}
We conclude by observing that
$$
\nabla_z P(t,z,\ve) = 2 U(t,z^2,\ve)z + 2 \vert z \vert^2 \bar{z} \left( \nabla_u U \right)(t,z^2,\ve)
$$
with $\nabla_u U = \tfrac{\partial U}{\partial u_1} + i \tfrac{\partial U}{\partial u_2}$ and
$$
\vert z \vert^2 z \nabla_z P(t,z,\ve) = 2 P(t,z,\ve)z^2 + 2 \vert z \vert^6 \left( \nabla_u U \right)(t,z^2,\ve).
$$
\end{proof}

The above lemma will play a role in the proof of the main theorem. Next we are going to describe how to produce a zero-energy closed orbit of the system \eqref{lc0} if we are given a generalized $T$-periodic solution of \eqref{eqmain0}. This converse process will not be employed in the proof of the main theorem but it is included to justify the definition of generalized solution.

As a preliminary observation note that the Hamiltonian function $\mathcal{K}_\ve$ is invariant under the involution
$$
\mathcal{J}: \mathbb{C}^2 \times \mathbb{T} \times \mathbb{R} \to \mathbb{C}^2 \times \mathbb{T} \times \mathbb{R}, \qquad
\mathcal{J}(z,w,t,\tau) = (-z,-w,t,\tau).
$$
From $\mathcal{K}_\ve \circ \mathcal{J} = \mathcal{K}_\ve$ we deduce that the set $\mathcal{K}_\ve^{-1}(0)$ is invariant under $\mathcal{J}$.

Let us now concentrate on the converse process mentioned above. Let $u(t)$ be a generalized $T$-periodic solution of \eqref{eqmain0}. From the classical estimates at collisions (see for instance \cite{Spe69}) we know that $\frac{1}{\vert u(t) \vert}$ is locally integrable. Define
$$
S = \int_0^T \frac{dt}{\vert u(t) \vert}.
$$
We are going to construct a solution $X(s)$ of \eqref{lc0} satisfying $\mathcal{K}_\ve(X(s)) = 0$ and one of the following conditions
\begin{itemize}
\item[(i)] $X(s+S) = X(s)$, for $s \in \mathbb{R}$,
\item[(ii)] $X(s+S) = \mathcal{J} X(s)$, for $s \in \mathbb{R}$.
\end{itemize}
In the second case $X(s)$ has period $2S$. The Hamiltonian system has no equilibria at $\mathcal{K}_\ve = 0$ and so we obtain a closed orbit.

The first step will be to define $X(s) = (z(s),w(s),t(s),\tau(s))$. The coordinate $t$ is defined with the help of Sundman integral
$$
s(t) = \int_0^t \frac{d\tau}{\vert u(\tau) \vert}.
$$
It defines a homeomorphism of $\mathbb{R}$ satisfying
$$
s(t+T) = s(t) + S.
$$
Then $t = t(s)$ is the inverse homeomorphism. At this moment we can only say that $t(s)$ is smooth on $\mathbb{R} \setminus Z^*$ with $Z^* = t^{-1}(Z)$. Note that $Z^*$ is a discrete set.

The definition of generalized solution implies that the function $t \in \mathbb{R} \setminus Z \mapsto \frac{u(t)}{\vert u(t) \vert} \in \mathbb{S}^1$ admits a continuous extension from $\mathbb{R}$ to $\mathbb{S}^1$ that is $T$-periodic. Then we can find a continuous argument. That is, a continuous function $\theta: \mathbb{R} \to \mathbb{R}$ satisfying
$$
\theta(t + T) = \theta(t) + 2\pi m, \qquad t \in \mathbb{R},
$$
for some integer $m$ and
$$
u(t) = \vert u(t) \vert e^{i\theta(t)}, \qquad t \in \mathbb{R}.
$$
From this identity we deduce that $\theta(t)$ is $\mathcal{C}^\infty$ on each interval contained in $\mathbb{R} \setminus Z$. Define
\begin{equation}\label{polar}
z(s) = \vert u(t(s)) \vert^{1/2} e^{i\theta(t(s))/2}, \qquad s \in \mathbb{R}.
\end{equation}
This is a continuous function that can be periodic or anti-periodic,
$$
z(s+S) = \pm z(s), \qquad s \in \mathbb{R},
$$
depending on whether $m$ is even or odd. We know that $z(s)$ is smooth outside $Z^*$. For each $s \in \mathbb{R} \setminus Z^*$ we define the remaining coordinates
$$
w(s) = 4z'(s), \qquad \tau(s) = - E(t(s)) + \ve \,U(t(s),z(s)^2,\ve)
$$
with
$$
E(t) = \frac12 \vert \dot u(t) \vert^2 - \frac{1}{\vert u(t) \vert}.
$$
The energy can be expressed ($u = z^2$) as
\begin{equation}\label{en}
E(t(s)) = \frac{1}{\vert z(s) \vert^2}\left( \frac{\vert w(s) \vert^2}{8} - 1 \right)
\end{equation}
and so our candidate to solution of \eqref{lc0} satisfies $\mathcal{K}_\ve(X(s)) = 0$ if $s \in \mathbb{R} \setminus Z^*$.

Let us know prove that $X(s)$ is a solution of \eqref{lc0} on each interval contained in $\mathbb{R} \setminus Z^*$. The first and third equations are direct consequences of the above definitions. To check the second equation we differentiate the identity $u(t(s)) = z(s)^2$ and proceed as in the proof of Lemma \ref{l1} to obtain
$$
\ddot u = 2 \frac{z z'' \vert z \vert^2 - z^2 \vert z' \vert^2}{\vert z \vert^6}.
$$
From equation \eqref{eqmain0}, valid for each $t \in \mathbb{R} \setminus Z$, 
$$
\vert z \vert^6 \ve \, \left( \nabla_u U \right)(t,z^2,\ve) = \frac12 z \vert z \vert^2 w' + z^2 \left( 1 - \frac18 \vert w \vert^2 \right)
$$
and the formula for the energy \eqref{en} leads to the second equation.

To check the fourth equation we observe that $E(t)$ is smooth on $\mathbb{R} \setminus Z$ and satisfies
$$
\dot E(t) = \ve \, \langle \nabla_u U(t,u(t),\ve), \dot u(t) \rangle.
$$
Differentiating with respect to $s$ the formula defining $\tau(s)$ we obtain
$$
\tau'(s) = \ve \, \vert z(s) \vert^2 \left( \partial_t U \right)(t(s),z(s)^2,\ve).
$$

The last step will be to prove that $X(s)$ admits a continuous extension to the whole real line. Then it is easy to conclude that $X(s)$ is a solution of \eqref{lc0} defined for all $s \in \mathbb{R}$.

We already know that the functions $z(s)$ and $t(s)$ are continuous in $\mathbb{R}$. The definition of generalized solution implies that $E(t(s))$ admits a continuous extension and so the same can be said for $\tau(s)$. To extend $w(s)$ we fix 
$s_0 \in Z^*$ and some $\delta > 0$ such that $z(s) \neq 0$ if $0 < \vert s - s_0 \vert \leq \delta$. From the second equation
(valid is $s \notin Z^*$) we know that $w(s)$ is $\mathcal{C}^1$ on the compact intervals $[s_0,s_0+\delta]$ and
$[s_0-\delta,s_0]$. Here we are using that $z(s)$, $t(s)$ and $\tau(s)$ are continuous at $s_0$. From
$\mathcal{K}_\ve(X(s)) = 0$ if $0 < \vert s - s_0 \vert < \delta$ we deduce that the right and left limits
$w(s_0 \pm 0)$ satisfy $\vert w(s_0 \pm 0) \vert^2 = \tfrac12$. In consequence
$$
\lim_{s \to s_0^{\pm}} \frac{z(s)}{\vert z(s) \vert} = \sqrt{2} w (s_0 \pm 0).
$$
On the other hand the identity \eqref{polar} implies that 
$$
\lim_{s \to s_0} \frac{z(s)}{\vert z(s) \vert} = e^{i\theta(t_0)/2}.
$$
Hence $w(s_0 + 0) = w(s_0 - 0)$ and the function $w(s)$ is continuous at $s = s_0$. 

\begin{remark}
Using once again the identity $u(t(s)) = z(s)^2$ it is possible to prove that generalized solutions of \eqref{eqmain0} satisfy a law of reflection of velocities, namely
$$
\lim_{t \to t_0^+} \frac{\dot u(t)}{\vert \dot u(t) \vert} = - \lim_{t \to t_0^-} \frac{\dot u(t)}{\vert \dot u(t) \vert} 
$$ 
if $t_0 \in Z$.
\end{remark}

\subsection{Periodic manifolds for the regularized Kepler problem}

In this second step we are concerned with system \eqref{lc} when $\varepsilon = 0$, that is 
\begin{equation}\label{hs0}
JX' = \nabla \mathcal{K}_0(X). 
\end{equation}
More precisely, we are going to prove that, for any integer $k \geq 1$, system \eqref{hs0} admits a compact non-degenerate periodic manifold $\Sigma_k$, at the energy level $\mathcal{K}_{0}=0$ and with corresponding closed orbits of index $1$.

\subsubsection{The periodic manifolds}\label{sec321}

First, we set
$$
\Lambda_k = \left\{ X = (z,w,t,\tau) \in \mathcal{K}_0^{-1}(0) \, : \, \tau = \tau_k \right\},
$$
where
\begin{equation}\label{deftk}
\tau_k = \left( \frac{\sqrt{2} k  \pi}{T}\right)^{2/3}.
\end{equation}
Notice that from the definition it immediately follows that $\Lambda_k$ is diffeomorphic to 
$\mathbb{S}^3 \times \mathbb{S}^1$. In particular, it is compact. 

We now prove that $\Lambda_k$ is filled by closed orbits of \eqref{hs0}, meaning that for any $X_0 = (z_0,w_0,t_0,\tau_0) \in \Lambda_k$, the solution $X(\cdot) = X(\cdot;X_0)$ of \eqref{hs0} with $X(0) = X_0$ is periodic. To this end, we first observe that, writing $X(s) = (z(s),w(s),t(s),\tau(s))$, it holds that
$\tau(s) \equiv \tau_0 = \tau_k$ and
\begin{equation}\label{z}
 z(s) = z_0 \cos(\omega_k s) + \frac{w_0}{4\omega_k} \sin(\omega_k s),\qquad \mbox{ where } 	\quad \omega_k = \left( \frac{\tau_k}{2}\right)^{1/2}.
\end{equation}
Then, $z(s)$ is periodic with minimal period $\sigma_k = 2\pi/\omega_k$ and
$$
\vert z(s) \vert^2 = \vert z_0 \vert^2 \cos^2(\omega_k s) + \frac{\langle z_0, w_0 \rangle}{2\omega_k}\cos(\omega_k s) \sin(\omega_k s) + \frac{\vert w_0 \vert^2}{16 \omega_k^2} \sin^2(\omega_k s), 
$$ 
so that
$$
\int_{0}^{k \sigma_k} \vert z(s) \vert^2 \,ds = \frac{k \pi}{\omega_k}\left( \vert z_0 \vert^2 + \frac{\vert w_0 \vert^2}{16 \omega_k^2}\right) = \frac{k \pi}{\omega_k \tau_k} = T.
$$
Therefore $t(k \sigma_k) = t(0) + T$, proving that the orbit $X(s)$ is closed with minimal period $k \sigma_k$ and 
index $1$. Notice that the period of the periodic solution $X(s;X_0)$ is independent of $X_0$; from now, we will denote it by
$$
S_k = k \sigma_k = \frac{2k\pi}{\omega_k} = 2k \pi \sqrt{\frac{2}{\tau_k}}.
$$ 

Finally, let us define $\Sigma_k$ as the subset of $M \times \,]0,\infty[\,$, $M = \mathbb{C}^2 \times \mathbb{T} \times \mathbb{R}$, defined by
$$
\Sigma_k = \Lambda_k \times \{S_k\}.
$$
Then the conditions (i) and (ii) in the definition of periodic manifold hold trivially.

\subsubsection{Non-degeneracy}

For each point $X_0 \in M$ the tangent space $T_{X_0}(M)$ will be identified to $\mathbb{C}^2 \times \mathbb{R}^2$. Tangent vectors will be denoted by $Y = (Y_1,Y_2,Y_3,Y_4)$ with $Y_1, Y_2 \in \mathbb{C}$ and $Y_3, Y_4 \in \mathbb{R}$. The linearized system along the solution $X(s;X_0)$ is
\begin{equation}\label{linsys}
J Y' = D^2 \mathcal{K}_0(X(s;X_0)) Y.
\end{equation}
When $X_0 \in \Lambda_k$ this linear system is periodic with period $S_k$ and Floquet theory is applicable. The monodromy operator at $X_0$ can be defined as
$$
P(Y(0)) = Y(S_k)
$$
for each $Y(s)$ solution of \eqref{linsys}.

The compact periodic manifold has dimension $4$ and the phase space $M$ has dimension $6$. To check the non-degeneracy 
of $\Sigma_k$ we apply the conclusions of Section \ref{sec2} and we must prove that the sub-matrix $\Gamma$ is not the identity. This will be proved as soon as we obtain, for each $X_0 \in \Lambda_k$, a vector $Y_\star \in T_{X_0}(\mathcal{K}_0^{-1}(0))$ such that $(\textnormal{Id} - P)Y_\star$ is not collinear with 
$J\nabla \mathcal{K}_0(X_0)$.

The tangent space of $\mathcal{K}_0^{-1}(0)$ can be described as 
$$
T_{X_0}(\mathcal{K}_0^{-1}(0)) = \left\{ Y \in \mathbb{C}^2 \times \mathbb{R}^2 \, : \,
2 \tau_k \langle z_0, Y_1 \rangle + \frac{1}{4} \langle w_0, Y_2 \rangle + \vert z_0 \vert^2 Y_4 = 0\right\}
$$
and we will prove that $Y_\star = (z_0,0,0,-2\tau_k)$ is a vector in the above conditions. This proof involves some computations and we need to express \eqref{linsys} in coordinates, 
$$
\left\{
\begin{array}{l}
\vspace{0.2cm}
Y_1' = \frac{1}{4}Y_2\\
\vspace{0.2cm}
Y_2' = - 2 \tau_k Y_1 - 2 z(s) Y_4\\
\vspace{0.2cm}
Y_3' = 2 \langle z(s), Y_1 \rangle \\
Y_4' = 0.
\end{array}
\right.
$$
Let $Y(s)$ be the solution with initial condition $Y(0) = Y_\star$.
We know that $P(Y_\star) = Y(S_k)$ and we are going to compute this vector. From the fourth equation $Y_4(s) = -2\tau_k$ and the first and second equations lead to 
$$
Y_1'' + \omega_k^2\, Y_1 = \tau_k z(s), \qquad Y_1(0) = z_0, \quad Y_1'(0) = 0.
$$
The solution of this Cauchy problem is
\begin{equation}\label{first}
Y_1(s) = z_0 \cos (\omega_k s) + \frac{\tau_k w_0}{8 \omega_k^3} \sin(\omega_k s) + \frac{\tau_k s}{2 \omega_k}
\left( z_0 \sin  (\omega_k s) - \frac{w_0}{4\omega_k} \cos  (\omega_k s)\right).
\end{equation}
Then 
$$
Y_1(S_k) = z_0 - \frac{k\pi}{2\omega_k}w_0.
$$
From $Y_2(s) = 4 Y_1'(s)$ we obtain
$$
Y_2(S_k) = \frac{4 k\pi \tau_k}{\omega_k} z_0.
$$
From the third equation
$$
Y_3(S_k) = 2 \int_0^{S_k} \langle z(s), Y_1(s) \rangle \,ds.
$$
To compute this integral we first observe that
$$
\int_0^{S_k} s \sin (\omega_k s) \cos (\omega_k s) \,ds = -\frac{k\pi}{2\omega_k^2}
$$
and that
$$
\int_0^{S_k} s \sin^2 (\omega_k s) \,ds = \int_0^{S_k} s \cos^2 (\omega_k s) \,ds.
$$ 
Then the identities \eqref{z} and \eqref{first} together with the standard properties of the inner product lead to
$$
Y_3(S_k) = \frac{k\pi}{\omega_k} \left( \vert z_0 \vert^2 + \frac{3}{16 \omega_k^2}\vert w_0 \vert^2\right)
$$
and using the fact that $\mathcal{K}_0(X_0) = 0$ together with $\tau(s) = \tau_k$ the above expression can be written as
$$
Y_3(S_k) = \frac{k\pi}{\omega_k} \left( - 2\vert z_0 \vert^2 + \frac{3}{\tau_k}\right)
$$
Now it is clear that the vectors
$$
Y(S_k) - Y(0) = \left( -\frac{k\pi}{2\omega_k}w_0, \frac{4k\pi \tau_k}{\omega_k}z_0, -\frac{2k\pi}{\omega_k} \left( \vert z_0 \vert^2 - \frac{3}{2 \tau_k}\right),0\right)
$$
and
$$
J \nabla \mathcal{K}_0(X_0) = \left( -\frac{w_0}{4}, 2\tau_k z_0, -\vert z_0 \vert^2, 0\right)
$$
are linearly independent over the real numbers.

\begin{remark}\label{non-deg}
It is not hard to prove that $\mu = 1$ is the only Floquet multiplier of system \eqref{linsys}. In consequence the algebraic multiplicity of $\mu = 1$ is $6$, a number greater than the dimension of $\Sigma$. This means that the notions of non-degeneracy introduced in \cite{Bot,Mos,Wei77} are not applicable to our problem. In this direction see also the remark in page 246 of \cite{Wei78}.
\end{remark}

\subsection{The perturbation argument}

We are now ready to conclude the proof by a direct application of the perturbation result stated in Section \ref{sec2}.

The symplectic form in $M = \mathbb{C}^2 \times \mathbb{T} \times \mathbb{R}$
$$
\omega = dx \wedge du + dy \wedge dv + dt \wedge d\tau
$$
where $z = x + iy$, $w = u + iv$, is exact with primitive $x du + y dv - \tau dt$. Moreover $E = 0$ is a regular value of $\mathcal{K}_0$ and $\Sigma_k$ is a compact non-degenerate periodic manifold. The category of $\mathbb{S}^3 \times \mathbb{S}^1$ will play a role to count the number of bifurcations. From our point of view it is sufficient to know that this category
is at least $2$ so that there is at least one branch of closed orbits emanating from each $\Sigma_k$. Let us fix any integer $l \geq 1$ and select some $\ve_1 > 0$ and neighborhoods $\mathcal{U}_k$ of $\Sigma_k$ in $M \times \mathbb{R}$, $k = 1,\ldots,l$, such that
\begin{equation}\label{safe}
\vert \tau - \ve U(t,z^2,\ve) - \tau_k \vert < \frac12 \min_{1 \leq k < h \leq l} (\tau_h - \tau_k)
\end{equation}
for each $(z,w,t,\tau) \in \mathcal{U}_k$ and $0 \leq \ve \leq \ve_1$. For $0 < \ve \leq \ve^*(l) < \ve_1$ small enough, a closed orbit to \eqref{lc} can be found lying on the energy level $\mathcal{K}_\ve^{-1}(0)$ and in $\mathcal{U}_k$, for any $k =1,\ldots,l$. Using Lemma \ref{l1} we obtain generalized periodic solutions of \eqref{eqmain0} which will be denoted by $u_k(t)$. In principle the period of these solutions will be a multiple of $T$. To complete the proof of Theorem \ref{main} we must prove that these solutions have indeed period $T$ and that they are different. Since for $\ve = 0$ the index of the closed orbit is $\eta = 1$ by construction and the index cannot change for small perturbations, we conclude that $u_k(t)$ has period $T$. To prove that $u_h(t)$ and $u_k(t)$ do not coincide if $1 \leq k < h \leq l$ we recall the formula expressing the energy of a generalized soluion $u(t)$ in terms of the associated $(z(s),w(s),t(s),\tau(s))$, namely
$$
E(t) := \frac12 \vert \dot u(t) \vert^2 - \frac{1}{\vert u(t) \vert} = -\tau(s(t)) + \ve U(t,z^2(s(t)),\ve).
$$
In view of \eqref{safe} we conclude that $\vert E_k(t) + \tau_k \vert < \frac12 (\tau_h-\tau_k)$ and 
$\vert E_h(t) + \tau_h \vert < \frac12 (\tau_h-\tau_k)$ and this implies that $E_k(t) \neq E_h(t)$ everywhere.

\section{Proof in the 3-d case}\label{sec5}

In this section, we give the proof for $N=3$. We are going to follow the discussion of the previous section, illustrating the main differences. Levi-Civita regularization was the main tool in two dimensions, now we will apply the Kustaanheimo-Stiefel regularization as described in \cite{Zha15}.

The skew-field of quaternions will be denoted by $\mathbb{H}$. Given $z = z_0 + z_1 i + z_2 j + z_3 k$ in $\mathbb{H}$, the real part will be denoted by $\Re(z) = z_0$ and the imaginary part by $\Im (z) = z_1 i + z_2 j + z_3 k$. Quaternions with vanishing real part are called purely imaginary and typically they will be denoted by $u$. The real vector space
$$
\mathbb{I} \mathbb{H} = \{ u \in \mathbb{H} \, : \, \Re (u) = 0\}
$$
has three dimensions and will play an important role.
Given an arbitrary quaternion $z \in \mathbb{H}$, the number $u = \bar{z} i z$ is purely imaginary. This fact motivates the definition of the map
$$
KS: \mathbb{H} \to \mathbb{I}{\mathbb{H}}, \qquad z \mapsto \bar{z} i z
$$
or, in coordinates, 
\begin{equation}\label{compute}
u = (z_0^2 + z_1^2 - z_2^2 - z_3^2) i + 2(z_1 z_2 - z_0 z_3) j + 2 (z_1 z_3 + z_0 z_2) k.
\end{equation}

In analogy with the previous section we consider the Hamiltonian function
$$
\mathcal{K}_\varepsilon (z,w,t,\tau) = \tau \vert z \vert^2 + \frac{\vert w \vert^2}{8}-1 + \varepsilon P(t,z,\varepsilon)
$$
where $P(t,z,\varepsilon) = \vert z \vert^2 U(t,\bar{z}iz,\varepsilon)$ and $U$ is meant as a function
$U: \mathbb{I}{\mathbb{H}} \to \mathbb{R}$. The symplectic manifold is 
$\mathbb{H} \times \mathbb{H} \times \mathbb{T} \times \mathbb{R}$ with the form
\begin{equation}\label{syf}
\omega = \sum_{h=0}^3 dz_h \wedge dw_h + dt \wedge d\tau.
\end{equation}
With the notation
$$
\nabla_z P = \partial_{z_0}P + i \, \partial_{z_i}P + j \, \partial_{z_2}P + k \, \partial_{z_3}P
$$
the Hamiltonian system is given again by \eqref{lc0} and \eqref{lc} but there are some differences with respect to the planar case. In contrast to Lemma \ref{l1}, in three dimensions not all closed orbits lying in $\mathcal{K}_\varepsilon^{-1}(0)$ will give rise to periodic solutions of \eqref{eqmain0}. We need a further condition.

\begin{lemma}\label{l1k}
For any $\varepsilon > 0$, any periodic solution $X(s)$ of \eqref{lc0}, lying on the energy level $\widetilde K_\varepsilon^{-1}(0)$ and such that
\begin{equation}\label{momentum}
\Re(\bar z(s) i w(s)) = 0, \quad \mbox{ for every } s,
\end{equation}
gives rise to a generalized $\eta T$-periodic solution of \eqref{eqmain0}, where $\eta$ is the index of $X(s)$. 
\end{lemma}

Before the proof we present an useful quaternionic formula. It can be obtained by combining the identity
\eqref{compute} with direct computations.

Assume that $F = F(z)$ and $G = G(u)$ are smooth functions with $F: \mathbb{H} \to \mathbb{R}$ and $G: \mathbb{I}\mathbb{H} \to \mathbb{R}$ and employ the notations
$$
\nabla_z F = \partial_{z_0}F + i \,\partial_{z_1}F + j \,\partial_{z_2}F + k \,\partial_{z_3}F
$$
$$
\nabla_u G = i \,\partial_{u_1}G + j \, \partial_{u_2}G + k\, \partial_{u_3}G.
$$
If we assume in addition that $F = G \circ KS$ then
\begin{equation}\label{uf}
\nabla_z F = - 2iz \nabla_u G.
\end{equation}

\begin{proof}
We define
$$
u(t) = KS(z(s(t)))
$$
where $s(t)$ is the inverse of $t(s)$. As in the planar case, we can prove that the condition (i) in the definition of generalized solution holds. Also the existence of limit of the collision direction in (iii) is proved in the same way. The existence of the limit of the collision energy follows from the identity
$$
\dot u = \frac{\bar z' i z + \bar z i z'}{\vert z \vert^2} \circ s =
\frac{2 \bar z i z'}{\vert z \vert^2} \circ s,
$$
where we have used \eqref{momentum}.
To prove (ii), we further compute
\begin{align*}
\ddot u & = 
 2 \left[\frac{(\bar z' i z' + \bar z i z'') \vert z \vert^2 - \bar z i z' (\bar z' z + \bar z z')}{\vert z \vert^6} \right] \circ s  \\
 & = 2 \left(\frac{\bar z i z'' \vert z \vert^2 - \bar z i z \vert z' \vert^2}{\vert z \vert^6}\right) \circ s + 
2 \left(\frac{\vert z \vert^2 \bar z' i z' - \bar z i z' \bar z z'}{\vert z \vert^6}\right) \circ s.
\end{align*}
Now, the first term can be computed using the differential equation and the zero-energy relationship, yielding
$$
2 \left(\frac{\bar z i z'' \vert z \vert^2 - \bar z i z \vert z' \vert^2}{\vert z \vert^6}\right) = - \frac{\bar z i z}{\vert z \vert^6} - 
\varepsilon\frac{\bar z i z P(t,z,\varepsilon) - \frac{\vert z \vert^2}{2}\bar z i \nabla_z P(t,z,\varepsilon)}{\vert z \vert^6}.
$$
In analogy with the planar case and taking into account \eqref{uf} we obtain
$$
\nabla_z P(t,z,\varepsilon) = 2 z U(t,\bar z i z,\varepsilon) - \vert z \vert^2 2 i z \nabla_u U(t,\bar z i z,\varepsilon)
$$
and
\begin{equation}\label{new11}
\vert z \vert^2 \bar z i \nabla_z P(t,z,\varepsilon) = 2 \bar z i z P(t,z,\varepsilon) + 2 \vert z \vert^6 \nabla_u U(t,\bar z i z,\varepsilon).
\end{equation}
On the other hand, the second term of $\ddot u$ can be rewritten as 
$$
2 \,\frac{\vert z \vert^2 \bar z' i z' - \bar z i z' \bar z z'}{\vert z \vert^6} = 
2 \,\frac{\bar z (z \bar z' i - i z' \bar z) z'}{\vert z \vert^6} = 
-\frac{\bar z \Re(i w \bar z) z'}{\vert z \vert^6}
$$
and we thus see that it vanishes in view of \eqref{momentum}.
Notice that $\Re (z \tilde z) = \Re (\tilde z z)$ for any quaternions $z$ and $\tilde z$.
\end{proof}

\begin{remark}
Assume now the $X(s)$ is a solution of the Hamiltonian system \eqref{lc0} that is not periodic but satisfies the conditions
$$
z(s+S) = \zeta(s) z(s), \quad w(s+S) = \zeta(s) w(s), \quad t(s+S) = t(s)+ \eta T, \quad \tau(s+S) = \tau(s), 
$$
where $\zeta: \mathbb{R} \to \mathbb{S}^1$ is any function. Here the unit circle has been embedded in $\mathbb{H}$ via the identification
$$
\mathbb{S}^1 = \{ z \in \mathbb{H} \, : \, z = z_0 + z_1 i, \, z_0^2 + z_1^2 = 1\}.
$$
Then the conclusion of the Lemma still holds and $u(t) = \overline{z(s(t))} i z(s(t))$ is $\eta T$-periodic.
\end{remark}

We observe that the function
$$
BL(X) = \Re (\bar z i w), \qquad X = (z,w,t,\tau),
$$
is a first integral of system \eqref{lc0}, as it can be easily verified directly, after having observed that, in view of \eqref{new11}, $\bar z i \nabla_z P(t,z,\ve)$ is a purely imaginary quaternion. For further convenience, we write below the explicit expression for $BL$, 
$$
BL(X) = -z_0 w_1 + z_1 w_0 - z_2 w_3 + z_3 w_2,
$$
from which one easily computes $\nabla BL(X) = (iw,-iz,0,0)$. Incidentally, notice also that $BL(X) = - \langle z, i w \rangle$.

In order to find solutions of \eqref{lc0} satisfying \eqref{momentum}, that is $BL(X(s)) = 0$ for every $s$, we will combine the result of Section \ref{sec2} with 
a procedure known as \emph{symplectic reduction}. Roughly speaking, we are going to prove that the Hamiltonian system
\eqref{lc0} naturally descends to an Hamiltonian system on an $8$-dimensional symplectic manifold $M_0$, obtained as the quotient of the level set $BL^{-1}(0)$ by a free circle action. The key point for this is to relate the first integral $BL$ with a suitable invariance of the Hamiltonian $\mathcal{K}_\ve$, via the concept of \emph{momentum map}.

We give below the details of this procedure, following the framework developed in the book \cite{CB}. We start with the group $G = \mathbb{S}^1$ and the action on the phase space $\mathbb{H} \times \mathbb{H} \times \mathbb{T} \times \mathbb{R}$,
$$
g \cdot X = (gz,gw,t,\tau)
$$
if $g \in \mathbb{S}^1$ and $X= (z,w,t,\tau) \in \mathbb{H} \times \mathbb{H} \times \mathbb{T} \times \mathbb{R}$. Notice that he Hamiltonian $\mathcal{K}_\varepsilon$ is invariant under this action because the numbers $g$ and $i$ commute. For points of the type $X = (0,0,t,\tau)$ the isotropy group is $\mathbb{S}^1$ and so the action is not free. For this reason we will work on the manifold
$$
M = \left( \mathbb{H}^2 \setminus \{0\}\right) \times \mathbb{T} \times \mathbb{R}
$$
where the action is free. It is also a proper action because the Lie group $G = \mathbb{S}^1$ is compact. The associated Lie algebra can be identified to the imaginary axis, $\mathfrak{g} = i \mathbb{R}$, so that the exponential map $\textnormal{exp}: \mathfrak{g} \to G$ coincides with the complex exponential. Given $\xi \in \mathfrak{g}$, the infinitesimal generator of the action in the direction of $\xi$ is given by 
$$
\mathfrak{X}^\xi(X) = \frac{d}{ds} \left[ e^{s\xi} \cdot (z,w,t,\tau)\right]_{\mid s = 0} = (\xi z, \xi w,0,0)
$$
for each $X= (z,w,t,\tau) \in M$. This is a Hamiltonian vector field with $\mathfrak{X}^\xi = J \nabla \Phi^\xi$ where
$$
\Phi^\xi: M \to \mathbb{R}, \qquad \Phi^\xi = - \Im(\xi) BL.
$$ 
After the identification of the dual $\mathfrak{g}^*$ with $\mathbb{R}$ we can say that $BL$ is the momentum map of this Hamiltonian action.

From now on $\Phi = BL$. In particular the domain of $BL$ is $M$ and
the real number $\mu = 0$ is a regular value. To apply the theorem of regular reduction in Chapter VII of \cite{CB}, we finally need to check the coadjoint equivariance of $\Phi$. In our case, the coadjoint action of $G = \mathbb{S}^1$ on $\mathfrak{g}^*$ is trivial because the group is commutative. In consequence the coadjoint equivariance of $\Phi$ just means
$$
BL(g \cdot X) = BL(X)
$$
for each $g \in \mathbb{S}^1$, $X \in M$. This property holds because $i$ and $g$ commute. 

Then, by the theorem of symplectic reduction the reduced space
$$
M_0 = \,\faktor{BL^{-1}(0)}{\mathbb{S}^1}
$$
is a symplectic manifold of dimension $8$. Points in $M_0$ are equivalence classes of the type
$$
\overline{X} = \overline{(z,w,t,\tau)} = \{ (gz,gw,t,\tau) \, : \, g \in \mathbb{S}^1\}
$$
where $X = (z,w,t,\tau) \in M$ with $\Re (\bar z i w) = 0$. To describe the tangent space of $M_0$ at a point $\overline{X}$
we recall that $(BL^{-1}(0),\mathbb{S}^1,M_0)$ is a principal $\mathbb{S}^1$-bundle over $M_0$ (see the definition in page 
315 of \cite{CB}). In particular the canonical projection $\pi_0: BL^{-1}(0) \to M_0$ is a surjective submersion and this implies that 
$$
T_{\overline{X}}(M_0) = (d\pi_0)_X \,T_X(BL^{-1}(0)) \cong  \,\faktor{T_X(BL^{-1}(0))}{\textnormal{ker}(d\pi_0)_X}\,.
$$
The tangent space of $BL^{-1}(0)$ at $X$ can be described as
$$
T_X(BL^{-1}(0)) = \{ Y = (Y_1,Y_2,Y_3,Y_4) \in \mathbb{H} \times \mathbb{H} \times \mathbb{R} \times \mathbb{R} \, : \, \langle \nabla BL(X),Y \rangle = 0 \}
$$
where $\langle \cdot, \cdot \rangle$ denotes the inner product in $\mathbb{H}^2 \times \mathbb{R}^2$ when it is identified to $\mathbb{R}^{10}$.
The space $\textnormal{ker}(d\pi_0)_X$ is one-dimensional and spanned by the vector $J \nabla BL(X) = (iz,iw,0,0)$. Vectors in $T_{\overline{X}}(M_0)$ will be described as equivalence classes in the quotient space,
$$
\overline{Y} = Y + \textnormal{ker}(d\pi_0)_X \,.
$$
The symplectic form $\omega_0$ in $M_0$ is obtained from the identity 
$$
\pi_0^{*} \,\omega_0 = i^* \omega
$$
where $i: BL^{-1}(0) \to M$ is the inclusion map. Note that $i^* \omega$ can be thought as the restriction of the symplectic form \eqref{syf} to the sub-manifold $BL^{-1}(0)$.

The reduced Hamiltonian is defined by
$$
\overline{\mathcal{K}}_\varepsilon: M_0 \to \mathbb{R}, \qquad \overline{\mathcal{K}}_\varepsilon(\overline{X}) = \mathcal{K}_\varepsilon(X)
$$
and the associated Hamiltonian vector field will be denoted by $\overline{\chi}_\varepsilon$. We are going to apply the result in Section \ref{sec2} on the energy level $\overline{\mathcal{K}}_\varepsilon^{-1}(0)$ of the system
\begin{equation}\label{wtrue}
\dot{\overline{X}} = \overline{\chi}_\varepsilon(\overline{X}), \qquad \overline{X} \in M_0.
\end{equation}
The orbits of this sytem can be lifted to orbits of the initial Hamiltonian system \eqref{lc0} lying on $BL^{-1}(0)$. This is a consequence of the identity
\begin{equation}\label{page17}
(d \pi_0)_{X} \, \chi_\varepsilon(X) = \overline{\chi}_\varepsilon(\pi_0(X)), \qquad X \in BL^{-1}(0),
\end{equation}
where $\chi_\varepsilon$ is the restriction of the vector field $J \nabla \mathcal{K}_\varepsilon$ to the submanifold $BL^{-1}(0)$. In principle the lift of a periodic solution of \eqref{wtrue} is not necessarily periodic but this should not create any trouble in view of the remark after Lemma \ref{l1k}.

To be in the framework of Section \ref{sec2}, we now claim that the form $\omega_0$ is exact on $M_0$. To prove this we first observe that the form given by \eqref{syf} is exact on $M$ with primitive
$$
\alpha = \sum_{h=0}^3 z_h dw_h - \tau dt.
$$
This $1$-form is invariant under the action, meaning that
$$
\alpha_{g \cdot X} (g \cdot Y) = \alpha_X (Y)
$$
for each $X = (z,w,t,\tau) \in M$ and $Y = (Y_1,Y_2,Y_3,Y_4) \in T_X(M) \equiv \mathbb{H}^2 \times \mathbb{R}^2$. To justify this identity we note that the inner product in $\mathbb{H}$ is invariant under the action of rotations and therefore
$$
\alpha_{g \cdot X} (g \cdot Y) =  \langle gz,g Y_2 \rangle - \tau Y_3 = \langle z,Y_2 \rangle - \tau Y_3 = 
\alpha_X (Y).
$$
The invariance of $\alpha$ is inherited by $i^* \alpha$, allowing to define a $1$-form $\alpha_0$ in $M_0$ via the identity
\begin{equation}\label{pull1}
\pi_0^{*} \,\alpha_0 = i^* \alpha
\end{equation}
or, equivalently, $(\alpha_0)_{\overline{X}}(\overline{Y}) = \alpha_X(Y)$. Taking differential in \eqref{pull1} we obtain
$$
\pi_0^*(d\alpha_0) = i^* (d\alpha) = i^*\omega = \pi_0^*\omega_0.
$$
Since $\pi_0$ is a submersion the identity $\pi_0^*(d\alpha_0-\omega_0) = 0$ implies that $d\alpha_0 - \omega_0 = 0$, proving that $\omega_0$ is exact.

As a second step we must also check that $E = 0$ is a regular value of $\overline{\mathcal{K}}_0$.
Since $\mathcal{K}_0 \vert_{BL^{-1}(0)} = \overline{\mathcal{K}}_0 \circ \pi_0$, we can use again that $\pi_0$ is a submersion and reduce the question to prove that $0$ is a regular value of the restriction of $\mathcal{K}_0$ to $BL^{-1}(0)$.
This is a consequence of the linear independence (over $\mathbb{R}$) of the vectors $\nabla BL(X)$ and $\nabla \mathcal{K}_0(X)$ when $\mathcal{K}_0(X) = 0$. Indeed if we consider the vectors in $\mathbb{H}^2$,
$$
\beta = (iw,-iz), \qquad k = (2\tau z, \frac14 w)
$$
with $\tau \vert z \vert^2 + \frac18 \vert w \vert^2 = 1$ then at least one of the four numbers $\tau z_0^2 + \frac18 w_1^2$,
$\tau z_1^2 + \frac18 w_0^2$, $\tau z_3^2 + \frac18 w_3^2$, $\tau z_3^2 + \frac18 w_2^2$ will not vanish. Then some of the $2 \times 2$ minors of the matrix with rows $\beta$ and $k$ will not vanish. This is a sub-matrix of the matrix with rows $\nabla BL(X)$ and $\nabla \mathcal{K}_0(X)$, which has rank two.

To construct the periodic manifold we first set
$$
\Lambda_k = \left\{ X = (z,w,t,\tau) \in \mathcal{K}_0^{-1}(0) \cap BL^{-1}(0) \, : \, \tau = \tau_k \right\},
$$
with $\tau_k$ given by \eqref{deftk}, and the very same computation therein shows that $\Lambda_k$ is filled by closed orbits
of the system in $BL^{-1}(0)$, $\dot X = \chi_\ve(X)$. As a consequence,
$$
\overline{\Lambda}_k = \pi_0 ( \Lambda_k ) \subset \overline{\mathcal{K}}_0^{-1}(0) \subset M_0
$$ 
is composed by periodic solutions of the system \eqref{wtrue}. Moreover these solutions are not constant because the third component $t(s)$ satisfies $t(s+S_k) = t(s) + T$ where $S_k$ is the quantity defined in Section \ref{sec321}. Clearly the sets $\Lambda_k$ and $\overline{\Lambda}_k$ are compact and we claim that they have a structure of smooth manifold with respective dimensions $7$ and $6$. Actually the function $f: BL^{-1}(0) \to \mathbb{R}^2$, $f(X) = (\mathcal{K}_0(X),\tau-\tau_k)$ is such that $\Lambda_k = f^{-1}(0)$. Similarly $\overline{\Lambda}_k= \overline{f}^{-1}(0)$ with $\overline{f} = f \circ \pi_0$. The independence of the vectors $\beta$ and $k$ discussed above can be employed to prove that $0$ is a regular value for both functions $f$ and $\overline{f}$.

The set
$$
\Sigma_k = \overline{\Lambda}_k \times \{S_k\}
$$
satisfies the conditions (i) and (ii) of Section \ref{sec2} and $\Sigma_k$ is a periodic manifold. Let us now prove that the non-degeneracy condition (iii) also holds. The phase space $M_0$ has dimension $8$ and the periodic manifold $\Sigma_k$ has dimension $6 = 8 - 2$. The same argument employed in the two-dimensional case leads us to reformulate (iii) as
\smallbreak
\noindent
\textit{First non-degeneracy condition:} for each $\overline{X} \in \overline{\Lambda}_k$ there exists a vector $\overline{Y}_\star \in T_{\overline{X}}(\overline{\mathcal{K}}_0^{-1}(0))$ such that
$(\textnormal{Id}-\overline{P})\overline{Y}_\star$ is not collinear with $\overline{\chi}_0(\overline{X})$.

\noindent
Here $\overline{P}: T_{\overline{X}}(M_0) \to T_{\overline{X}}(M_0)$ denotes the corresponding monodromy operator associated to 
$\dot{\overline{X}} = \overline{\chi}_0(\overline{X})$.
\smallbreak
To work on quotient spaces is more delicate and for this reason we reformulate the above condition in terms of the original flow,
\smallbreak
\noindent
\textit{Second non-degeneracy condition:} for each $X \in \Lambda_k$ there exists a vector $Y_\star \in T_X(BL^{-1}(0)) \cap T_X(\mathcal{K}_0^{-1}(0))$ such that $(\textnormal{Id}-P)Y_\star$ is not in the two-dimensional space spanned by
$J \nabla BL(X)$ and $J \nabla \mathcal{K}_0(X)$; that is,
$$
(\textnormal{Id }-P)Y_\star \notin \textnormal{Sp}\{J \nabla BL(X), J \nabla \mathcal{K}_0(X) \}.
$$
Here $P: T_X(BL^{-1}(0)) \to T_X(BL^{-1}(0))$ is the monodromy operator associated to the system $\dot X = \chi_0(X)$. Note that the identity \eqref{page17} implies that
\begin{equation}\label{l17}
\overline{P} \circ (d\pi_0)_X = (d\pi_0)_X \circ P \quad \mbox{ for each } X \in \Lambda_k.
\end{equation}
We also observe that
\begin{equation}\label{ker}
(d\pi_0)_X \left(\textnormal{ker}(d\mathcal{K}_0) \cap T_X(BL^{-1}(0))\right) = \textnormal{ker}(d\overline{\mathcal{K}}_0)_{\overline{X}}.
\end{equation}
This is a consequence of chain rule and the surjective character of $(d\pi_0)_X$.
\smallbreak
For our purposes it is sufficient to check that the second non-degeneracy condition implies the first. Assume that $Y_\star$ is given by the second condition. Then $\overline{Y}_\star = (d\pi_0)_X Y_\star$ belongs to $\textnormal{ker}(d\overline{\mathcal{K}}_0)_X = T_{\overline{X}}(\overline{\mathcal{K}}_0^{-1}(0))$. This is a consequence of \eqref{ker}. Assume now by a contradiction argument that $(\textnormal{Id}-\overline{P})\overline{Y}_\star$ were collinear with $\overline{\chi}_0(\overline{X}) = (d\pi_0)_X \chi_0(X)$. Since $\chi_0(X)$ is the restriction of the vector field $J \nabla \mathcal{K}_0$, there should exist some $\lambda \in \mathbb{R}$ such that $(\textnormal{Id}-\overline{P})\overline{Y}_\star =
\lambda (d\pi_0)_X(J\nabla \mathcal{K}_0(X))$. From \eqref{l17} we deduce that $(d\pi_0)_X (\textnormal{Id }-P)Y_\star = 
(d\pi_0)_X(\lambda J \nabla \mathcal{K}_0(X))$. The space $\textnormal{ker}(d\pi_0)_X$ is spanned by $J\nabla BL(X)$ so that there should exist $\mu \in \mathbb{R}$ such that 
$$
(\textnormal{Id}-\overline{P})\overline{Y}_\star =
\lambda J \nabla \mathcal{K}_0(X) + \mu J \nabla BL(X).
$$
This would imply that the second condition fails. In consequence the first condition is valid whenever the second holds.

We are going to prove that $\Sigma_k$ is non-degenerate using the second condition and taking advantage of the result in the
2d-case. Let us fix $X_0 \in \Lambda_k$. The crucial observation is that
$$
\Re(\bar z_0 i w_0) = 0
$$
because $X_0 = (z_0,w_0,t,\tau) \in BL^{-1}(0)$. This implies the existence of a Levi-Civita plane $\Pi \subset \mathbb{H}$ with $z_0,w_0 \in \Pi$. We recall that a Levi-Civita plane $\Pi \subset \mathbb{H}$ is a two dimensional linear space spanned by vectors $v_1,v_2 \in \mathbb{H}$ with $\Re(\bar v_1 i v_2) = 0$. Given such a plane, the set
$$
E_\Pi = \Pi \times \Pi \times \mathbb{R}^2
$$
is invariant under the flow of $J\dot X = \nabla \mathcal{K}_0(X)$. Moreover there exists an isomorphism of vector spaces $\Pi \cong \mathcal{C}$ such that the flow on $\mathcal{K}_0^{-1}(0)$ becomes the Levi-Civita regularized flow in dimension $2$. That is, system \eqref{lc0} with $\ve = 0$. See \cite{Zha15} for more details. As a consequence, the six-dimensional space $E_\Pi$ is invariant under the monodromy operator and, up to a linear conjugacy, the restriction of $P$ to $E_\Pi$ is nothing but
the monodromy map for the 2d-case. Hence, by the arguments in the 2d-case, there exists $Y_\star \in E_\Pi \cap T_X(\mathcal{K}_0^{-1}(0))$ such that $(\textnormal{Id}-P)Y_\star$ is not collinear with $J\nabla \mathcal{K}_0(X_0)$. 
To conclude, we thus need to prove that $Y_\star \in T_X(BL^{-1}(0))$ and that 
$(\textnormal{Id}-P)Y_\star \notin \textnormal{Sp}\{ J \nabla \mathcal{K}_0(X_0),J \nabla BL(X_0)\}$. As for the first fact, we write 
$Y_\star = (Y_1,Y_2,Y_3,Y_4)$ with $Y_1,Y_2 \in \Pi$ and we simply compute
$$
\langle \nabla BL(X_0), Y_\star \rangle = - \langle i w_0,Y_1 \rangle + \langle i z_0,Y_2 \rangle = 
\Re(\overline{w}_0 i Y_1) - \Re(\overline{z}_0 i Y_2)
$$
which vanishes since $w_0,z_0,Y_1,Y_2$ lies in the Levi-Civita plane $\Pi$. A very similar computation shows that
$J\nabla BL(X_0)$ is orthogonal to $E_\Pi$, thus implying also the second fact.

The rest of the proof has no substantial differences with the 2d-case. Now we apply Lemma \ref{l1k}.

\section{Miscellaneous remarks}\label{sec6}

\subsection{Bifurcation from infinity}

We say that the equation \eqref{eqmain0} has a periodic bifurcation from infinity if there exists $\ve_* > 0$ such that for $0 < \ve < \ve_*$ there exists a $T$-periodic solution $u_\ve(t)$ with
$$
\min_t \vert u_\ve(t) \vert \to \infty \quad \mbox{ as } \ve \to 0^+.
$$
Note that these solutions do not have collisions for small $\ve$.

In contrast to the bifurcation from periodic manifolds this type of bifurcation does not appear for all perturbations. We illustrate this phenomenon on the special class of equations
\begin{equation}\label{fkp}
\ddot u = - \frac{u}{\vert u \vert^3} + \ve \,p(t)
\end{equation}
where $p: \mathbb{R} \to \mathbb{R}^N$ is a $\mathcal{C}^\infty$ and $T$-periodic function and $N \geq 2$.

\begin{proposition}\label{aver}
The equation \eqref{fkp} has a periodic bifurcation from infinity if and only if
\begin{equation}\label{nonav}
\int_0^T p(t)\,dt \neq 0.
\end{equation}
\end{proposition}

\begin{proof}
Assume first that \eqref{nonav} holds. Then we perform the change of variables
$$
x = \varepsilon^{1/2} u
$$	
so as to obtain the equation
$$
\ddot x = \varepsilon^{3/2} \left( - \frac{x}{\vert x \vert^3} + p(t) \right).
$$
Writing this equation as the first order system
$$
\dot x = \varepsilon^{3/4} y, \qquad \dot y = \varepsilon^{3/4} \left( -\frac{x}{\vert x \vert^3} + p(t)\right)
$$
we see that the averaging method (see for instance Theorem 6.4 in \cite{Hale}) applies, giving a $T$-periodic solution $(x,y)$ bifurcating from
$(x^*,0)$, where $x^*$ is the unique solution of
$$
\frac{x^*}{\vert x^* \vert^3} = \frac{1}{T}\int_0^T p(t)\,dt.
$$
Since $u = \varepsilon^{-1/2}x$, the corresponding $T$-periodic solution of \eqref{fkp} bifurcates from infinity when $\varepsilon \to 0^+$.

We give some details on the previous application of the averaging method. First consider the map
$$
\sigma: (\mathbb{R}^N \setminus \{0\}) \times \mathbb{R}^N \to 
\mathbb{R}^N \times \mathbb{R}^N, \qquad \sigma(x,y) = \left( y, - \frac{x}{\vert x \vert^3} + \bar{p} \right)
$$
with $\bar{p} = \tfrac{1}{T} \int_0^T p(t)\,dt$. For $\bar{p} \neq 0$ this map has the unique zero $(x^*,0)$ with $x^* = \frac{1}{\vert \bar{p} \vert^{3/2}} \bar{p}$ and we must check that this zero is non-degenerate; that is, $\textnormal{det} [ \sigma'(x,y)(x^*,0) ] \neq 0$. The Jacobian matrix is easily computed, 
$$
\sigma'(x,y) = \left( \begin{array}{cc}
0 & \textnormal{Id}_N \\
S_N & 0 
\end{array} \right)
$$
where $S_N = \vert x \vert^{-5} (- \vert x \vert^2 \textnormal{Id}_N + 3 x \otimes x)$ and $x \otimes x$ is the $N \times N$ matrix whose components are $(x \otimes x)_{ij} = x_i x_j$. The formula 
$\textnormal{det} (a \textnormal{Id}_N + b z \otimes z) = a^{N-1} (a + b \vert z \vert^2)$ can be applied to obtain 
$\textnormal{det} [ \sigma'(x,y) ] = 2 \vert x \vert^{-3N} \neq 0$.

To prove the converse assume the existence of a sequence $\ve_n \searrow 0$ such that the system \eqref{fkp} has a $T$-periodic solution $u_n(t)$ for $\ve = \ve_n$ with $\min_t \vert u_n(t) \vert \to \infty$. It is not restrictive to assume $\ve_n < 1$ and $\vert u_n(t) \vert \geq 1$ for each $t \in \mathbb{R}$ and $n \geq 1$. From the equation \eqref{fkp} we deduce that $\Vert \ddot u_n \Vert_\infty$ is bounded, namely $\vert \ddot u_n(t) \vert \leq 1 + \Vert p \Vert_\infty$ everywhere. Here $\Vert \cdot \Vert_\infty$ denotes the usual norm in $L^{\infty}(\mathbb{R})$. The periodicity of $u_n(t)$ and this bound implies that
\begin{equation}\label{Rb}
\vert u_n(t) - u_n(0) \vert \leq R, \qquad t \in \mathbb{R}, n \geq 1,
\end{equation}
where $R$ is a constant which only depends upon $\Vert p \Vert_\infty$ and $T$. After extracting a subsequence we assume that the sequence of unit vectors $\frac{u_n(0)}{\vert u_n(0) \vert}$ converges to some vector $\eta \in \mathbb{R}^N$ with $\vert \eta \vert = 1$. Associated to this vector we consider the two cones
$$
C_i = \left\{ x \in \mathbb{R}^N \, : \, \langle x, \eta \rangle \geq \alpha_i \vert x \vert \right\}, \qquad i = 1,2
$$
where $0 < \alpha_1 < \alpha_2 < 1$. We observe that $C_2$ is contained in $C_1$. Moreover there exists $\mu > 0$ such that if $\vert x \vert \geq \mu$ and $x \in C_2$ then the ball of center $x$ and radius $R$ is contained in $C_1$. Here $R$ is the constant given in \eqref{Rb}.

For large $n$ the vector $\frac{u_n(0)}{\vert u_n(0) \vert}$ must enter into $C_2$ and the same can be said about $u_n(0)$. Assuming that $\vert u_n(0) \vert \geq \mu$ we deduce from \eqref{Rb} that $u_n(t)$ belongs to $C_1$. Hence for $n$ large enough,
$$
\frac{u_n(t)}{\vert u_n(t) \vert^3} \in C_1 \quad \mbox{ for every } t \in \mathbb{R}.
$$
This implies that $\int_0^T \frac{u_n(t)}{\vert u_n(t) \vert^3} \,dt \neq 0$. After integrating the equation \eqref{fkp} over a period we deduce that the condition \eqref{nonav} holds.
\end{proof}

\subsection{Removal of collisions}

Given a perturbed Kepler problem and a periodic solution with collisions, it seems reasonable to expect that collisions will disappear by slight changes in the perturbation. We present a very preliminary result in this direction. It is concerned with the problem
\begin{equation}\label{for}
\ddot u = - \frac{u}{\vert u \vert^3} + p(t)
\end{equation}
in dimension $N = 2$.

\begin{proposition}\label{Acol}
Assume that $p \in \mathcal{C}^\infty \left( \mathbb{R}/T\mathbb{Z},\mathbb{R}^2\right)$ is such that \eqref{for} has a $T$-periodic solution $u(t)$ with collisions. Then there exists sequences $p_n: \mathbb{R} \to \mathbb{R}^2$, $T_n > 0$, $u_n(t)$ solutions of \eqref{for} for $p(t) = p_n(t)$ such that
\begin{itemize}
\item[i)] $p_n$ is $\mathcal{C}^\infty$ and $T_n$-periodic with $T_n \to T$
\item[ii)] $u_n(t)$ is $T_n$-periodic, $u_n(t) \to u(t)$ uniformly on compact intervals, $u_n(t) \neq 0$ for each $t$
\item[iii)] $\int_I \vert p_n(t) - p(t) \vert \,dt \to 0$ as $n \to \infty$, for each bounded interval $I \subset \mathbb{R}$.
\end{itemize}
\end{proposition}

\begin{proof}
We shall assume that $u(t)$ has a single collision on each period. The case of multiple collisions can be treated with similar arguments; note that along a periodic solution the collisions are isolated, thus only finitely many of them lie on a compact interval. From now one we assume that
$$
u(0) = 0 \quad \mbox{ and } \quad u(t) \neq 0 \; \mbox{ if } 0 < \vert t \vert \leq \frac{T}{2}.
$$
We follow the discussion at the end of Section \ref{sec41} with $\ve = 1$ and $P(t,z) = \vert z \vert^2 \langle p(t), z^2 \rangle$. A solution $X(s)$ of \eqref{lc0} associated to $u(t)$ can be constructed. Let us recall that in this constructions there are two alternatives, either $X(s+S) = X(s)$ or $X(s+S) = \mathcal{J} X(s)$ with 
$$
T = \int_0^S \vert z(s) \vert^2 \,ds.
$$
We shall assume that we are in the first case. The second can be treated similarly. For $X(s) = (z(s),w(s),t(s),\tau(s))$ we know that $t = t(s)$ is a homeomorphism of the real line with $t(0) = 0$ and $t(s+S) = t(s) + T$. The inverse homeomorphism is denoted by $s = s(t)$. It satisfies $s(0) = 0$, $s(t+T) = s(t) + S$ and $u(t) = z(s(t))^2$. Let $\Psi: \mathbb{R} \to \mathbb{R}$ be a $\mathcal{C}^\infty$-function with the properties
$$
0 \leq \Psi \leq 1 \; \mbox{ everywhere}, \quad \Psi = 1 \; \mbox{ on } [-1,1], \quad \Psi(s) = 0 \; \mbox{ if } \vert s \vert \geq 2.
$$
Given $\mu \in \left(0,\tfrac{S}{4}\right)$ we define the function $z_{\mu}: \mathbb{R} \to \mathbb{R}^2$ as the $S$-periodic extension of 
$$
z_{\mu}(s) = z(s) + \mu^3 \Psi\left( \frac{s}{\mu}\right) v \quad \mbox{ if } \vert s \vert \leq \frac{S}{2}
$$
where $v$ is an unit vector in $\mathbb{R}^2$ which is orthogonal to $z'(0)$. Note that $\mathcal{K}_1(X(0)) = 0$ implies that $\vert z'(0) \vert^2 = \tfrac12$. 

A first observation is that $z_{\mu}(s)$ does not vanish when $\mu$ is small enough. To prove this we first note that $z(0) = z''(0) = 0$ so that $z(s) = z'(0)s + R(s)$ with $\vert R(s) \vert \leq C_1 \vert s \vert^3$ if $\vert s \vert \leq \tfrac{S}{2}$. Then, using that $v$ and $z'(0)$ are orthogonal,
\begin{equation}\label{low}
\begin{aligned}
\vert z_{\mu}(s) \vert & \geq \left\vert z'(0)s + \mu^3
\Psi\left( \frac{s}{\mu}\right)v  \right \vert - \vert R(s) \vert \\
& \geq \sqrt{\frac12 s^2 + \mu^6 \Psi\left( \frac{s}{\mu}\right)^2} - C_1 \vert s \vert^3.
\end{aligned}
\end{equation}
Since $s = 0$ is the only collision of $z(s)$ on $[-\tfrac{S}{2},\tfrac{S}{2}]$
there exists $C_1^* > 0$ such that $\vert z(s) \vert \geq C_1^* \vert s \vert$ if $\vert s \vert \leq \tfrac{S}{2}$. Now
it is easy to prove that there exists $C_2 > 0$ such that, for small $\mu$,
\begin{equation}\label{ES1}
\vert z_{\mu}(s) \vert \geq C_2 (\vert s \vert + \mu^3) \quad \mbox{ if } \vert s \vert \leq \frac{S}{2}.
\end{equation}
For $\vert s \vert \leq \mu$ we apply \eqref{low} and for $\mu \leq \vert s \vert \leq \tfrac{S}{2}$ we observe that
$$
\vert z_{\mu}(s) \vert \geq \vert z(s) \vert - \mu^3\geq C_1^* \vert s \vert - \mu^3 \geq \frac{C_1^*}{2}\vert s \vert + 
\frac{C_1^*}{2} \mu - \mu^3.
$$

It is clear that $z_{\mu}(s) \to z(s)$ when $\mu \to 0$ and this convergence is uniform in $s$. We define the number
$$
T_{\mu} = \int_0^S \vert z_{\mu}(s) \vert^2 \,ds
$$
and consider the diffeomorphism of the real line
$$
t_{\mu}(s) = \int_0^s \vert z_{\mu}(\sigma) \vert^2 \,d\sigma.
$$
It satisfies
$$
t_{\mu}(s+S) = t_{\mu}(s) + T_{\mu}
$$
and the inverse diffeomorphism $s_{\mu} = t_{\mu}^{-1}$ converges to $s = t^{-1}$ (as $\mu \to 0$) uniformly on compact intervals. This inverse satisfies
$$
s_{\mu}(t + T_{\mu}) = s_{\mu}(t) + S
$$
and the function $u_{\mu}(t) = z_{\mu}(s_{\mu}(t))^2$ is $T_{\mu}$-periodic and has no collisions. Moreover
$$
u_{\mu}(t) \to u(t) \quad \mbox{ as } \mu \to 0
$$
uniformly on compact intervals. Following Section \ref{sec41} we define
$$
\widetilde{p}_{\mu} = 2 \,\frac{z_{\mu} z''_{\mu}}{\vert z_{\mu} \vert^4} + \frac{z_{\mu}^2 ( 1 - 2 \vert z'_{\mu} \vert^2)}{\vert z _{\mu} \vert^6}.
$$
This function is $\mathcal{C}^\infty$ and $S$-periodic and some computations show that $u_{\mu}(t)$ is a $T_{\mu}$-periodic solution of
$$
\ddot u = - \frac{u}{\vert u \vert^3} + p_{\mu}(t)
$$
where $p_{\mu} = \widetilde p_{\mu} \circ s_{\mu}$.

To complete the proof it remains to find a sequence $\mu_n$ tending to zero such that the condition iii) holds
for $p_n = p_{\mu_n}$. After the change of variables $t = t(s)$ the integral in iii) is transformed into
\begin{equation}\label{trans}
\int_J \vert \widetilde{p}_n(s_n(t(s))) - \widetilde p(s)\vert \vert z(s) \vert^2 \,ds
\end{equation}
where $J$ is a bounded interval, $\widetilde p_n = \widetilde p_{\mu_n}$ and $s_n = s_{\mu_n}$.
For simplicity in the notation we assume that $J = [-\tfrac{S}{2},\tfrac{S}{2}]$ but the argument is the same for any other interval. From the definition of $p_{\mu}$ we know that $\widetilde p_n$ converges to $\widetilde p = p \circ t$ uniformly on any compact set $K \subset [-\tfrac{S}{2},\tfrac{S}{2}] \setminus \{0\}$. Given $s$ with $0 < \vert s \vert \leq \tfrac{S}{2}$ we know that $s_n(t(s)) \to s$ and therefore 
$$
\widetilde p_n(s_n(t(s))) \to \widetilde p(s) \quad \mbox{ as } n \to \infty
$$
if $0 < \vert s \vert \leq \tfrac{S}{2}$. To prove that the integral \eqref{trans} tends to zero it is enough to apply dominated convergence if we prove that there exists $C_3 > 0$ such that
\begin{equation}\label{ES2}
\vert \widetilde p_n(s_n(t(s))) \vert \vert z(s) \vert^2 \leq C_3 \quad \mbox{ if } \vert s \vert \leq \frac{S}{2}.
\end{equation}
The rest of the proof will be a sequence of estimates whose aim is to obtain \eqref{ES2}.

The first estimate is almost automatic,
\begin{equation}\label{ES3}
\vert z(s) \vert \leq C_4 \vert s \vert \quad \mbox{ if } \vert s \vert \leq \frac{S}{2}.
\end{equation}

Next we present an \emph{auxiliary result}: given $C_5 > 0$ there exists $\Delta > 0$ such that if $\sigma$ and $s$ are numbers with $\vert s \vert \leq \Delta$, $s \sigma > 0$ and 
$$
\vert \sigma^3 + 6 \mu^6 \sigma - s^3 \vert \leq C_5 s^4
$$
then $\frac{s^2}{\sigma^2 + \mu^6} \leq 16$.
To prove this auxiliary result we assume without loss of generality that $s > 0$ and we distinguish two cases. If $0 < s \leq 2\mu^3$ then $\frac{s^2}{\sigma^2 + \mu^6} \leq \frac{4\mu^6}{\sigma^2 + \mu^6} \leq 4$. If $s > 2\mu^3$ we consider the polynomial $p(x) = x^3 + 6 \mu^6 x - s^3$ and evaluate it at $x = \tfrac{s}{4}$ and $x = \sigma$, 
$$
p\left( \frac{s}{4}\right) = \left( \frac{1}{4^3} - 1\right)s^3 + \frac{3\mu^6}{2} s \leq  \left( \frac{1}{4^3} - 1\right) s^3 + \frac{3}{8}s^3 = -\frac{39}{64}s^3
$$
$$
p(\sigma) \geq - C_5 s^4.
$$
We deduce that $p(\sigma) > p\left( \tfrac{s}{4}\right)$ if $s$ is small. Since $p(x)$ is an increasing function we conclude that $\sigma > \tfrac{s}{4}$ and so $\frac{s^2}{\sigma^2 + \mu^6} \leq 16$. We will apply the above result to estimate the quantity $\sigma_n(s) = s_n(t(s))$. 

Our next estimate is the following:
\begin{equation}\label{ES4}
\vert \sigma_n(s) \vert \leq C_6 \vert s \vert \quad \mbox{ if } \vert s \vert \leq \frac{S}{2}.
\end{equation}
Indeed 
$$
\sigma_n'(s) = s_n'(t(s))t'(s) = \tfrac{1}{\vert z_n(s) \vert^2}{\vert z(s) \vert^2}
$$
and from the mean value theorem
$$
\sigma_n(s) = \sigma_n(s) - \sigma_n(0) = \frac{\vert z(\xi) \vert^2}{\vert z_n(\xi) \vert^2} s
$$
where $\xi$ lies between $s$ and $0$. From \eqref{ES1} and \eqref{ES3}, 
$$
\frac{\vert \sigma_n(s) \vert}{\vert s \vert} \leq \frac{C_4^2 \vert \xi \vert^2}{C_2^2 (\vert \xi \vert + \mu_n^3)^2} \leq \left(\frac{C_4}{C_2} \right)^2 =: C_6.
$$

We claim that
\begin{equation}\label{ES5}
\frac{s^2}{\sigma_n(s)^2 + \mu_n^6} \leq 16 \quad \mbox{ if } \vert s \vert \leq \frac{S}{2}.
\end{equation}
Since $\sigma_n$ converges to the identity uniformly in $\vert s \vert \leq \tfrac{S}{2}$ it is enough to obtain the estimate on a small neighborhood of $s = 0$. From the definition of $\sigma_n$ we have the identity
\begin{equation}\label{idet}
t_n(\sigma_n(s)) = t(s).
\end{equation}
We expand $t_n(s)$ and $t(s)$ in a neighborhood of $s = 0$. From $z(s) = z'(0)s + O(s^3)$ we deduce that $\vert z(s) \vert^2 = \tfrac{1}{2} s^2 + O(s^4)$ and $t(s) = \int_0^s \vert z(\sigma) \vert^2 \,d\sigma = \tfrac16 s^3 + O(s^5)$. Since $\Psi$ is flat at $\xi = 0$,
$\Psi(\xi) = 1 + O(\xi^3)$. Then $z_n(s) = z(s) + \mu_n^3 \Psi\left( \frac{s}{\mu_n}\right) v = z'(0) s + \mu_n^3 v + O(s^3)$ and $\vert z_n(s) \vert^2 = \tfrac12 s^2 + \mu_n^6 + O(s^3)$. The expansion of $t_n$ is
$$
t_n(s) = \int_0^s \vert z_n(\sigma)\vert^2 \,d\sigma = \frac16 s^3 + \mu_n^6 s + O(s^4). 
$$
The above expansion together with \eqref{idet} lead to
$$
\frac16 \sigma_n(s)^3 + \mu_n^6 \sigma_n(s) + O(\sigma_n(s)^4) = \frac16 s^3 + O(s^5).
$$
From \eqref{ES4}, 
$$
\sigma_n(s)^3 + 6 \mu_n^6 \sigma_n(s) - s^3 = O(s^4)
$$
and the auxiliary result implies that \eqref{ES5} holds.

We are now in a position to conclude. First, we prove that
\begin{equation}\label{ES6}
\frac{\vert z_n''(\sigma_n(s)) \vert}{\vert z_n(\sigma_n(s)) \vert^3} \vert z(s) \vert^2 \leq C_9 \quad \mbox{ if } \vert s \vert \leq \frac{S}{2}.
\end{equation}
From $z_n''(s) = z''(s) + \mu_n\Psi''\left( \tfrac{s}{\mu_n}\right)v$ we deduce that
$$
\vert z_n''(s) \vert \leq \vert z''(s) \vert + C_7 \vert s \vert \leq C_8 \vert s \vert.
$$
Here we have used $z''(0) = 0$ and that $\Psi''(\xi) = O(\xi)$. From \eqref{ES3}, \eqref{ES1} and \eqref{ES5}, 
$$
\frac{\vert z_n''(\sigma_n(s)) \vert}{\vert z_n(\sigma_n(s)) \vert^3} \vert z(s) \vert^2 \leq \frac{C_8 C_4^2 \vert \sigma_n(s) \vert s^2}{C_2^3 (\vert \sigma_n(s) \vert + \mu_n^3)^3} \leq C_9.
$$

Finally, we have
$$
\frac{\vert  1 - 2 \vert z_n'(\sigma_n(s)) \vert^2 \vert}{\vert z_n(\sigma_n(s)) \vert^4} \vert z(s) \vert^2 \leq C_{10} \quad \mbox{ if } \vert s \vert \leq \frac{S}{2}.
$$
From $\Psi'(\xi) = O(\xi^2)$, $z_n'(s) = z'(s) + \mu_n^2\Psi'\left( \tfrac{s}{\mu_n}\right)v = z'(0) + O(s^2)$. Then $\vert z_n'(s) \vert^2 = \tfrac12 + O(s^2)$ and $1 - 2 \vert z_n'(\sigma_n(s)) \vert^2 = O (\sigma_n(s)^2) = O(s^2)$. The conclusion follows as for \eqref{ES6}.
\end{proof}

\subsection{Perturbations with singularity: an example}

In the final section of the paper \cite{F1928} Fatou considered the equations of motion of a particle under the force of attraction of a rotating body. The body has certain symmetry properties and, in particular, it is symmetric with respect to the equator $\{ z = 0 \}$. The motion of the particle is constrained to this plane. 

\begin{figure}[!h]
\centering
\includegraphics[scale=0.35]{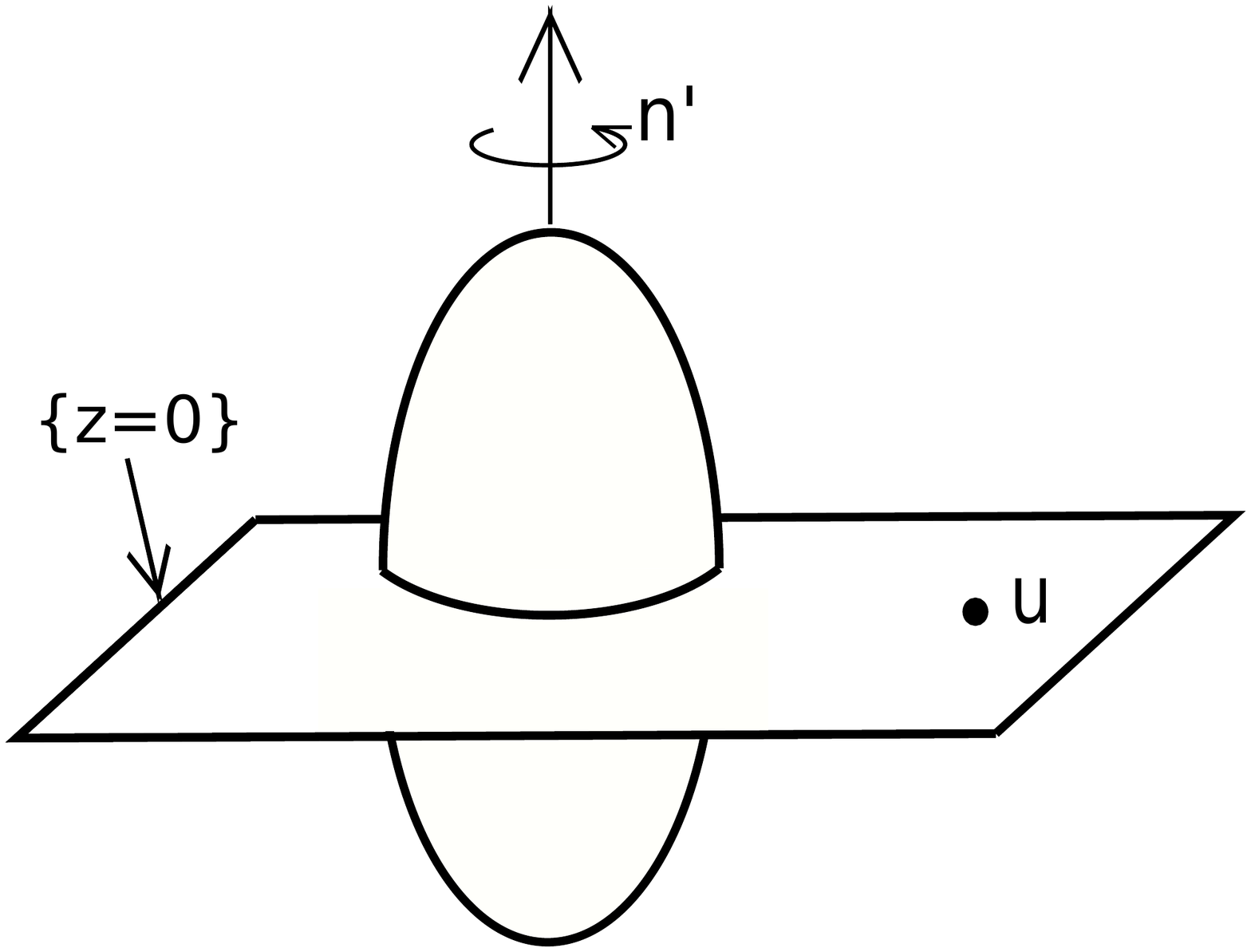}
\label{fig1}
\end{figure}

Standard considerations in Potential Theory allow to approximate the gravitational force acting on the point $u = (r\cos\theta,r\sin\theta,0)$ by the gradient of the function
\begin{equation}\label{pot}
\frac{1}{r} + \frac{1}{r^3} [k + h \cos (2(\theta-\beta))]
\end{equation}
where $\beta = n' t + \gamma$, $k > 0$ depends on the geometry of the body ($k = 0$ for a sphere) and $h> 0$ takes into account the unequal distribution of mass on meridians.

Fatou proved that the circular solutions of the unperturbed problem 
($k = h = 0$) have a certain stability property when the rotation of the body is very fast. Namely, for an angular velocity $n' \to \infty$ there exists a solution of the perturbed equation which remains close to the circular orbit for very large (but finite) intervals of time. When the body is close to a sphere and the distribution of mass is almost uniform, the parameters $k = k' \ve $ and $h = h' \ve$ are small and we find an equation of the type \eqref{eqmain0} with $N = 2$ and 
$$
U(t,u,\ve) = \frac{k'}{\vert u \vert^3} + \frac{h'}{\vert u \vert^5}
\left[ (u_1^2-u_2^2) \cos(2(n't + \gamma)) + 2 u_1 u_2 \sin(2(n't + \gamma))\right].
$$
This suggests the use of the averaging method to study the existence and stability of periodic solutions (see \cite{BO} for more details), however our main result Theorem \ref{main} does not apply. The perturbation $U$ has a singularity of high order and it is not clear how to define generalized solutions. In this context it is perhaps worth to recall that the approximation given by the formula \eqref{pot} is only valid on the free space and so the mechanical significance of the equations is lost at the singularity. 

\bigbreak
\medbreak
\noindent \textbf{Acknowledgements.}
Alberto Boscaggin acknowledges the support of the ERC Advanced Grant 2013 n. 339958
``Complex Patterns for Strongly Interacting Dynamical Systems - COMPAT'' and of the INDAM-GNAMPA Project ``Dinamiche complesse per il problema degli $N$-centri''. Rafael Ortega is partially supported by Spanish MINECO and ERDF project MTM2014-52232-P.  Lei Zhao is partially supported by NSFC No. 11601242, Fundamental Research Funds for the Central Universities of China, DFG FR 2637/2-1 and ZH 605/1-1.

\vspace{1 cm}
\normalsize

Authors' addresses:
\bigbreak
\medbreak
\indent Alberto Boscaggin \\
\indent Dipartimento di Matematica, Universit\`a di Torino, \\
\indent Via Carlo Alberto 10, I-10123 Torino, Italy \\
\indent e-mail: alberto.boscaggin@unito.it 

\bigbreak
\medbreak
\indent Rafael Ortega \\
\indent Departamento de Matem\'atica Aplicada, Universidad de Granada, \\
\indent E-18071 Granada, Spain \\
\indent e-mail: rortega@ugr.es 

\bigbreak
\medbreak
\indent Lei Zhao \\
\indent Institute of Mathematics, University of Augsburg, \\
\indent Universit\"atsstrasse 2, D-86159 Augsburg, Germany \\
\indent Chern Institute of Mathematics, Nankai University, China \\
\indent e-mail: lei.zhao@math.uni-augsburg.de

\end{document}